\documentclass[11pt]{article}
\usepackage{amsfonts, amssymb, amsmath, amsthm, xcolor, graphicx, pstricks, pstricks-add}
\begin{document}
\newcommand{\m}{\textnormal{mult}}
\newcommand{\ra}{\rightarrow}
\newcommand{\la}{\leftarrow}
\renewcommand{\baselinestretch}{1.1}

\theoremstyle{plain}
\newtheorem{thm}{Theorem}[section]
\newtheorem{cor}[thm]{Corollary}
\newtheorem{con}[thm]{Conjecture}
\newtheorem{cla}[thm]{Claim}
\newtheorem{lm}[thm]{Lemma}
\newtheorem{prop}[thm]{Proposition}
\newtheorem{example}[thm]{Example}

\theoremstyle{definition}
\newtheorem{dfn}[thm]{Definition}
\newtheorem{alg}[thm]{Algorithm}
\newtheorem{prob}[thm]{Problem}
\newtheorem{rem}[thm]{Remark}

\renewcommand{\baselinestretch}{1.1}

\title{\bf Generalized $D$-graphs for Nonzero Roots of the Matching Polynomial}
\author{
Cheng Yeaw Ku
\thanks{ Department of Mathematics, National University of Singapore, Singapore 117543. E-mail: matkcy@nus.edu.sg} \and K.B. Wong \thanks{
Institute of Mathematical Sciences, University of Malaya, 50603 Kuala Lumpur, Malaysia. E-mail:
kbwong@um.edu.my.} } \maketitle

\begin{abstract}\noindent
Recently, Bauer et al. (J Graph Theory 55(4) (2007), 343--358) introduced a graph operator $D(G)$, called the $D$-graph of $G$, which has been useful in investigating the structural aspects of maximal Tutte sets in $G$ with a perfect matching. Among other results, they proved a characterization of maximal Tutte sets in terms of maximal independent sets in the graph $D(G)$ and maximal extreme sets in $G$. This was later extended to graphs without perfect matchings by Busch et al. (Discrete Appl. Math. 155 (2007), 2487--2495). Let $\theta$ be a real number and $\mu(G,x)$ be the matching polynomial of a graph $G$. Let $\textnormal{mult} (\theta, G)$ be the multiplicity of $\theta$ as a root of $\mu(G,x)$. We observe that the notion of $D$-graph is implicitly related to $\theta=0$. In this paper, we give a natural generalization of the $D$-graph of $G$ for any real number $\theta$, and denote this new operator by $D_{\theta}(G)$, so that $D_{\theta}(G)$ coincides with $D(G)$ when $\theta=0$. We prove a characterization of maximal $\theta$-Tutte sets which are $\theta$-analogue of maximal Tutte sets in $G$. In particular, we show that for any $X \subseteq V(G)$, $|X|>1$, and any real number $\theta$, $\m(\theta, G \setminus X)=\m(\theta, G)+|X|$ if and only if $\m(\theta, G \setminus uv)=\m(\theta, G)+2$ for any $u, v \in X$, $u \not = v$, thus extending the preceding work of Bauer et al. and Busch et al. which established the result for the case $\theta=0$. Subsequently, we show that every maximal $\theta$-Tutte set $X$ is matchable to an independent set $Y$ in $G$; moreover, $D_{\theta}(G)$ always contains an isomorphic copy of the subgraph induced by $X \cup Y$. To this end, we introduce another related graph $S_{\theta}(G)$ which is a supergraph of $G$, and prove that $S_{\theta}(G)$ and $G$ have the same Gallai-Edmonds decomposition with respect to $\theta$. Moreover, we determine the structure of $D_{\theta}(G)$ in terms of its Gallai-Edmonds decomposition and prove that $D_{\theta}(S_{\theta}(G))=D_{\theta}(G)$.
\end{abstract}

\bigskip\noindent
{\sc keywords:} matching polynomial, Gallai-Edmonds decomposition, Tutte sets, extreme sets, $D$-graphs

\section{Introduction}

All the graphs in this paper are simple and finite. The vertex set and edge set of a graph $G$ will be denoted by $V(G)$ and $E(G)$ respectively.

\begin {dfn}\label{Intro}  An $r$-\emph {matching} in a graph $G$ is a set of $r$ edges, no two of which have a vertex in common. The number of $r$-matchings in $  G$ will be denoted by $p( G,r)$. We set $p(G,0)=1$ and define the \emph {matching polynomial} of $G$ by
\begin {equation}
\mu ( G,x)=\sum_{r=0}^{\lfloor n/2\rfloor} (-1)^rp(G,r)x^{n-2r}.\notag
\end {equation}
Let $u\in V(G)$. The graph obtained from $G$ by deleting the vertex $u$ and all edges that contain $u$ will be denoted by $G\setminus u$. Inductively, if $u_1,\dots, u_k\in V(G)$ then $G\setminus u_1\dots u_k=(G\setminus u_1\dots u_{k-1})\setminus u_k$. Note that the order in which the vertices are being deleted is not important, that is, if $i_1,\dots, i_k$ is a permutation  of $1,\dots, k$, we have  $G\setminus u_1\dots u_k= G\setminus u_{i_1}\dots u_{i_k}$. Furthermore if  $X=\{u_1,\dots, u_k\}$, we write $G\setminus X=G\setminus u_1\dots u_k$. Similarly, if $H$ is a subgraph of $G$ and $V(H)=\{v_1,\dots, v_k\}$, we write $G\setminus H=G\setminus v_1\dots v_k$.

Let $e_1,e_2,\dots, e_k\in E(G)$. We shall denote the graph obtained from $G$ by deleting the edges $e_1,e_2,\dots, e_k$ by $G-e_1e_2\dots e_k$.
\end {dfn}

It is well known that all roots of $\mu(G,x)$ are real. Throughout, let $\theta$ be a real number and $\textnormal {mult} (\theta, G)$ denote the multiplicity of $\theta$ as a root of $\mu(G,x)$. In particular, $\textnormal {mult} (\theta, G)=0$ if and only if $\theta$ is not a root of $\mu(G,x)$, if and only if $G$ has a perfect matching. In the literature, $\m(0,G)$ is also known as the {\em deficiency} of $G$, usually denoted by $\textnormal{def}(G)$.

\begin {lm}\label{interlacing}\textnormal{\cite[Corollary 1.3 on p. 97]{G0} (Interlacing)} Let $  G$ be a graph  and $u\in V(  G)$. Then
\begin {equation}
\textnormal {mult} (\theta,   G)-1\leq \textnormal {mult} (\theta,   G\setminus u)\leq \textnormal {mult} (\theta,  G)+1.\notag
\end {equation}
\end {lm}

\noindent As a consequence of Lemma \ref  {interlacing}, we can classify the vertices in a graph with respect to $\theta$ as follows:

\begin {dfn}\label {P:D3}\textnormal {\cite [Section 3]{G}} For any $u\in V(  G)$,
\begin {itemize}
\item [(a)] $u$ is $\theta$-\emph {essential} if $\textnormal {mult} (\theta,   G\setminus u)=\textnormal {mult} (\theta,   G)-1$,
\item [(b)] $u$ is $\theta$-\emph {neutral} if $\textnormal {mult} (\theta,   G\setminus u)=\textnormal {mult} (\theta,   G)$,
\item [(c)] $u$ is $\theta$-\emph {positive} if $\textnormal {mult} (\theta,   G\setminus u)=\textnormal {mult} (\theta,   G)+1$.
\end {itemize}
Furthermore if $u$ is not $\theta$-essential but is adjacent to some $\theta$-essential vertex, we say that $u$ is $\theta$-{\em special}.
\end {dfn}

The subgraph of $G$ induced by $\theta$-essential vertices plays an important role in the Gallai-Edmonds decomposition of $G$. Indeed, it consists of components such that every vertex is $\theta$-essential in each of the components. Such a component is called $\theta$-{\em critical}. It is worth noting that a connected graph is factor-critical if and only if it is $0$-critical.

Recently, a graph operator $D(G)$, called the $D$-graph of $G$, was introduced by Bauer et al. \cite{BBMS} for graphs with a perfect matching. This notion was later extended to general graphs by Busch et al. \cite{BFK}.

\begin {dfn}\label{D-graph}
Let $G$ be a graph. The graph $D(G)$ is defined as follows:
\begin {itemize}
\item[(a)] $V(D(G))=V(G)$, and
\item[(b)] $(x,y) \in E(D(G))$ if and only if $\m(0, G \setminus xy) \le \m(0, G)$.
\end {itemize}
\end {dfn}

Let $X$ be a subset of $V(G)$. Recall that $X$ is a {\em Tutte set} in $G$ if $\omega_{o}(G\setminus X)=\m(0, G)+|X|$, where $\omega_{o}(G)$ denotes the number of odd components of $G$. Another standard term for Tutte set in the literature is {\em barrier} (see \cite{Lo}). If $\m(0, G \setminus X) = \m(0,G)+|X|$, we say that $X$ is an {\em extreme set} in $G$.

The following theorem summarizes the main structural result in \cite{BBMS} and \cite{BFK}:

\begin {thm}\label{Tutte}
Let $G$ be a graph and $X \subseteq V(G)$, $\vert X\vert>1$. The followings are equivalent:
\begin {itemize}
\item[(a)] $X$ is a maximal Tutte set in $G$,
\item[(b)] $X$ is a maximal extreme set in $G$,
\item[(c)] $X$ is a maximal independent set in $D(G)$.
\end {itemize}
\end {thm}

The above has proven useful in investigating maximal Tutte sets. For example, it has been instrumental in determining the complexity of finding maximum Tutte sets for several interesting classes of graphs \cite{BBKMSS}.

To generalize the preceding result for nonzero real $\theta$, we need a $\theta$-analogue of $D(G)$. The following is a natural generalization of $D(G)$ for general $\theta$:

\begin {dfn}\label{D-graph-2}
Let $G$ be a graph and $\theta$ be a real number. The graph $D_{\theta}(G)$ is defined as follows:
\begin {itemize}
\item[(a)] $V(D_{\theta}(G))=V(G)$, and
\item[(b)] $(x,y) \in E(D_{\theta}(G))$ if and only if $\m(\theta, G \setminus xy) \le \m(\theta, G)$.
\end {itemize}
\end {dfn}

We also require a $\theta$-analogue of Tutte sets and extreme sets. The corresponding definitions were first introduced in \cite{KW}:

\begin {dfn}
Suppose $X \subseteq V(G)$.
\begin {itemize}
\item[(a)] $X$ is a $\theta$-Tutte set if $c_{\theta}(G \setminus X) = \m(\theta, G) + |X|$, where $c_{\theta}(G)$ denotes the number of $\theta$-critical components of $G$.
\item[(b)] $X$ is a $\theta$-extreme set if $\m(\theta, G \setminus X) = \m(\theta, G)+|X|$.
\end {itemize}
\end {dfn}

Note that the definitions of $0$-extreme set and extreme set coincide. But the definitions of $0$-Tutte
set and Tutte set are different. Nevertheless, the definition of a $\theta$-Tutte set is not unmotivated. Indeed, it is motivated by a $\theta$-analogue of Berge's formula proved by the authors in \cite{KW}. Interested readers may refer to \cite{KW} for a more detailed description of $\theta$-Tutte sets and $\theta$-extreme sets.

One of our main results is the following:

\begin {thm}\label{Main01}
Let $G$ be a graph, $X \subseteq V(G)$, $|X|>1$, and $\theta$ be a real number. The followings are equivalent:
\begin {itemize}
\item[(a)] $X$ is a maximal $\theta$-Tutte set in $G$,
\item[(b)] $X$ is a maximal $\theta$-extreme set in $G$,
\item[(c)] $\m(\theta, G \setminus uv)=\m(\theta, G)+2$ for any $u, v \in X$, $u \not = v$.
\end {itemize}
\end {thm}

It is clear that conditions (b) of Theorem \ref{Main01} and Theorem \ref{Tutte} are the same when $\theta=0$. In fact, we shall see later that conditions (a) and (c) of Theorem \ref{Main01} and Theorem \ref{Tutte} are also equivalent when $\theta=0$. Therefore, Theorem \ref{Main01} can be regarded as an extension of Theorem \ref{Tutte} to general $\theta$.

In this paper, we introduce another related graph $S_{\theta}(G)$ which is a supergraph of $G$ obtained by joining any $\theta$-special vertex to all the other vertices in $G$. Note that if $G$ has no $\theta$-special vertices then $S_{\theta}(G)=G$. We shall establish the followings:

\begin {thm}\label{Main03}
Let $G$ be a graph and $\theta$ be a real number. Then $G$ and $S_{\theta}(G)$ have the same Gallai-Edmonds decomposition.
\end {thm}

\begin {thm}\label{Main04} If $G$ and $G'$ have the same Gallai-Edmonds decomposition with respect to $\theta$, then $D_{\theta}(G)\cong D_{\theta}(G')$. In particular, $D_{\theta}(G)=D_{\theta}(S_{\theta}(G))$.
\end {thm}

It was also proved in \cite{BBMS} and \cite{BFK} that $D(G)$ contains an isomorphic copy of $G$. In general, $D_{\theta}(G)$ does not contain an isomorphic copy of $G$. However, we can prove the following:

\begin {thm}\label{Main02}
Given any $\theta$-extreme set $X$ of $G$ with $|X|>1$, there exists an independent set $Y$ disjoint from $X$ such that $X$ is matchable to $Y$ and $D_{\theta}(G)$ contains an isomorphic copy of the subgraph of $G$ induced by $X \cup Y$.
\end {thm}
Recall that a set $X$ is {\em matchable} to a set $Y$ if there is a matching in $G$ which matches every vertex of $X$ to a vertex in $Y$.

The $D$-graph $D(G)$ demonstrates interesting properties when iterated, in particular, it converges very quickly regardless of the structure of the underlying graph $G$, that is $D(D(D(G))) \equiv D(D(G))$ (see \cite{BBMS, BFK}). At the present, we do not know whether such property also holds for the $D_{\theta}$-operator.

The outline of this paper is as follows.

In Section 2, we list some basic properties of the matching polynomial and describe the Gallai-Edmonds decomposition for general root $\theta$ which is an important tool for the rest of the paper. Theorem \ref{Main01} is proved in Section 3. In Section 4, we prove Theorem \ref{Main03} which consequently allow us to establish Theorem \ref{Main04} in Section 5. Finally, in Section 6, we relate $\theta$-extreme sets with matchings and independent sets and prove Theorem \ref{Main02}.

\section{Gallai-Edmonds Decomposition}

The followings  are some basic properties of $\mu (G,x)$.

\begin {thm}\label{basic_property} \textnormal { \cite[Theorem 1.1 on p. 2] {G0}}
\begin {itemize}
\item [(a)] $\mu (  G\cup  H,x)=\mu (  G,x)\mu (  H,x)$ where $  G$ and $  H$ are disjoint graphs,
\item [(b)] $\mu (  G,x)=\mu (  G-e,x)-\mu (  G\setminus uv, x)$ if $e=(u,v)$ is an edge of $  G$,
\item [(c)] $\mu (  G,x)=x\mu (  G\setminus u,x)-\sum_{i\sim u} \mu (  G\setminus ui,x)$ where $i\sim u$ means $i$ is adjacent to $u$,
\item [(d)] $\displaystyle \frac {d}{dx} \mu (  G,x)=\sum_{i\in V(  G)} \mu (  G\setminus i,x)$ where $V(  G)$ is the vertex set of $  G$.
\end {itemize}
\end {thm}

Note that if $\textnormal {mult} (\theta,   G)=0$ then for any $u\in V(G)$, $u$ is either $\theta$-neutral or $\theta$-positive and no vertices in $G$ can be $\theta$-special. By \cite[Corollary 4.3]{G}, a $\theta$-special vertex is $\theta$-positive. Therefore
\begin {equation}
V(G)=B_{\theta}(G)\cup A_{\theta}(G)\cup P_{\theta}(G)\cup N_{\theta}(G),\notag
\end {equation}
where
\begin {itemize}
\item [] $B_{\theta}(G)$ is the set of all $\theta$-essential vertices in $G$,
\item [] $A_{\theta}(G)$ is the set of all $\theta$-special vertices in $G$,
\item [] $N_{\theta}(G)$ is the set of all $\theta$-neutral vertices in $G$,
\item [] $P_{\theta}(G)$ is the set of all $\theta$-positive vertices which are not $\theta$-special in $G$,
\end {itemize}
is a partition of $V(G)$.
Note that there are no $0$-neutral vertices. So $N_0(G)=\varnothing$ and $V(G)=B_{0}(G)\cup A_{0}(G)\cup P_{0}(G)$.

The Gallai-Edmonds Structure Theorem (henceforth the GEST) contains structural information of the above decomposition of $V(G)$ with respect to the root $\theta=0$ of $\mu (G,x)$. In \cite {KC}, Chen and Ku extended the GEST to any root $\theta$. It essentially consists of two lemmas: the $\theta$-Stability Lemma and the $\theta$-Gallai's Lemma.

\begin {thm}\label {P:T5}\textnormal {\cite [Theorem 1.5]{KC} (The $\theta$-Stability Lemma)}\\ Let $G$ be a graph with $\theta$ a root of $\mu (G,x)$. If $u\in A_{\theta} (G)$ then
\begin {itemize}
\item [(i)] $B_{\theta}(G\setminus u)=B_{\theta}(G)$,
\item [(ii)] $P_{\theta}(G\setminus u)=P_{\theta}(G)$,
\item [(iii)] $N_{\theta}(G\setminus u)=N_{\theta}(G)$,
\item [(iv)] $A_{\theta}(G\setminus u)=A_{\theta}(G)\setminus \{u\}$.
\end {itemize}
\end {thm}

\begin {thm}\label {P:T6}\textnormal {\cite [Theorem 1.7]{KC} (The $\theta$-Gallai's Lemma)} \\ If $G$ is connected and $\theta$-critical then $\textnormal {mult} (\theta,   G)=1$.
\end {thm}

By Theorem \ref {P:T5} and Theorem \ref {P:T6}, it is straightforward to deduce the following whose proof is omitted.

\begin {cor}\label {P:C7}\
\begin {itemize}
\item [(i)]  $A_{\theta}(G\setminus A_{\theta}(G))=\varnothing$, $B_{\theta}(G\setminus A_{\theta}(G))=B_{\theta}(G)$, $P_{\theta}(G\setminus A_{\theta}(G))=P_{\theta}(G)$, and $N_{\theta}(G\setminus A_{\theta}(G))=N_{\theta}(G)$.
\item [(ii)] $G\setminus A_{\theta}(G)$ has exactly $\vert A_{\theta}(G)\vert+\textnormal {mult} (\theta, G)$ $\theta$-critical components.
\item [(iii)] If $H$ is a component of $G\setminus A_{\theta}(G)$ then either $H$ is $\theta$-critical or $\textnormal {mult} (\theta, H)=0$.
\item [(iv)] The subgraph induced by $B_{\theta}(G)$ consists of all the  $\theta$-critical components in $G\setminus A_{\theta}(G)$.
\end {itemize}
\end {cor}

\section{The Structure of Maximal $\theta$-Tutte Sets}

In this section, we study the structure of maximal $\theta$-Tutte sets. We first establish a characterization of these sets in their relation to $\theta$-extreme sets.

Let $X \subseteq V(G)$. By interlacing (Lemma \ref{interlacing}), it is immediate that $\m(\theta, G \setminus X) \le \m(\theta, G) + |X|$. On the other hand, by the $\theta$-Gallai's Lemma (Theorem \ref{P:T6}), we have $c_{\theta}(G \setminus X) \le \m(\theta, G \setminus X)$. Therefore, if $X$ is a $\theta$-Tutte set, then it is also $\theta$-extreme. The converse is not true. Nevertheless, a maximal $\theta$-extreme set is always a maximal $\theta$-Tutte set.

\begin {thm}\label{max-Tutte}
Let $G$ be a graph and $\theta$ be a real number. A set $X$ is a maximal $\theta$-Tutte set in $G$ if and only if $X$ is a maximal $\theta$-extreme set in $G$.
\end {thm}

\begin{proof}
It remains to show that if $X$ is a maximal $\theta$-extreme set in $G$, then $c_{\theta}(G \setminus X) = \m(\theta, G) + |X|$. Notice that $G \setminus X$ has no $\theta$-positive vertices; otherwise, any $\theta$-positive vertex of $G \setminus X$ together with $X$ form a larger $\theta$-extreme set containing $X$, violating the maximality of $X$. In particular, $D_{\theta}(G \setminus X) \cup N_{\theta}(G \setminus X) = V(G) \setminus X$. This means that if $H_{1}, \ldots, H_{s}$ are the components of $G-X$ with $\theta$ as a root, then $V(H_{1}) \cup \cdots \cup V(H_{s}) = D_{\theta}(G \setminus X)$. By GEST for $\theta$, each $H_{j}$ is $\theta$-critical and satisfies $\m(\theta, H_{j})=1$. Since $X$ is $\theta$-extreme, we obtain
\[ \m(\theta, G)+|X| = \m(\theta, G \setminus X) = \sum_{j=1}^{s} \m(\theta, H_{j}) = s = c_{\theta}(G \setminus X), \]
and thus $X$ is a $\theta$-Tutte set in $G$.

If $X$ is not a maximal $\theta$-Tutte set in $G$, then $X$ is properly contained in a $\theta$-Tutte set $Y$. But $Y$ would be a $\theta$-extreme set which properly contains $X$, violating the maximality of $X$. Hence, $X$ is a maximal $\theta$-Tutte set.
\end{proof}

It is worth noting that a $0$-Tutte set is always a Tutte set but the converse is not true (\cite[Proposition 2.3]{KW}). However, a maximal Tutte set  is always a maximal $0$-Tutte set (\cite[Proposition 2.4]{KW}).

We proceed to prove another characterization of maximal $\theta$-Tutte sets.

Recall that $D_{\theta}(G)$ is completely determined by knowing the multiplicities of $\theta$ when deleting any two distinct vertices of $G$. Moreover, by interlacing, these multiplicities lie between $\m(\theta, G)-2$ and $\m(\theta, G)+2$. This motivates the following terminology:

\begin {dfn}\label {BP:D3}  Let $G$ be a graph. We define the graph $D_{r,\theta}(G)$ for $r=-2,-1,0,1,2$ as follows:
\begin {itemize}
\item [(a)] $V(D_{r,\theta}(G))=V(G)$, and
\item [(b)] $e=(u,v)\in E(D_{r,\theta}(G))$ if and only if $\textnormal {mult} (\theta, G\setminus uv)=\textnormal {mult} (\theta, G)+r$.
\end {itemize}
\end {dfn}

Note that $D_{\theta}(G) = D_{-2, \theta}(G) \cup D_{-1, \theta}(G) \cup D_{0, \theta}(G)$. Note also that the powers of $x$'s in the matching polynomial $\mu (G,x)$ are either all even or all odd. This implies that $\theta$ is a root of $\mu (G,x)$ if and only if $-\theta$ is. Also the powers of $x$'s in the $n$-th derivative of $\mu (G,x)$ are either all even or all odd. From these we deduce that $\textnormal {mult} (\theta,G)=\textnormal {mult} (-\theta,G)$. Hence we have

\begin {thm}\label {BP:T4a} $D_{r,\theta}(G)=D_{r,-\theta}(G)$.
\end {thm}

In view of Theorem \ref{Main01}, we further introduce the following definition:

\begin {dfn}\label{nice_set_independent} A set $X \subseteq V(G)$ with $|X|>1$ is said to be $\theta$-{\em nice} in $G$ if $\m(\theta, G \setminus uv)=\m(\theta, G)+2$ for any $u, v \in X$, $u \not = v$. Clearly, if $X$ is $\theta$-nice then all its vertices are $\theta$-positive. Equivalently, a set $X$ is $\theta$-nice in $G$ if the subgraph $D_{2,\theta}(G)[X]$ of $D_{2, \theta}(G)$ induced by $X$ is complete.
\end {dfn}

It has been shown that $X$ is an $0$-extreme set if and only if $X$ is an independent set of $D_{0}(G)$, provided that $\vert X\vert>1$ (Theorem \ref{Tutte}). Recall that $N_{0}(G)=\varnothing$ and so $\textnormal {mult} (0,G\setminus uv)=-2,0,2$ for all $u,v\in V(G)$ (by interlacing (Lemma \ref{interlacing})). This implies that $X$ is an independent set of $D_{0}(G)$ if and only if $D_{2,0}(G)[X]$ is a complete graph. Hence, Theorem \ref{Tutte} can be reformulated as follows:

\begin {thm}\label{extreme_zero}
Let $G$ be a graph and $X \subseteq V(G)$, $|X|>1$. The following are equivalent:
\begin {itemize}
\item[(a)] $X$ is a maximal $0$-Tutte set in $G$,
\item[(b)] $X$ is a maximal extreme set in $G$,
\item[(c)] $X$ is a maximal complete subgraph in $D_{2,0}(G)$, that is $X$ is a maximal $0$-nice set in $G$.
\end {itemize}
\end {thm}

So it is quite natural to ask whether Theorem \ref{extreme_zero} holds for $\theta\neq 0$. Indeed, we shall prove that

\begin {thm}\label{theta-max-Tutte}
Let $G$ be a graph, $X \subseteq V(G)$, $|X|>1$ and $\theta$ be a real number. The followings are equivalent:
\begin {itemize}
\item[(a)] $X$ is a maximal $\theta$-Tutte set in $G$,
\item[(b)] $X$ is a maximal $\theta$-extreme set in $G$,
\item[(c)] $X$ is a maximal complete subgraph in $D_{2,\theta}(G)$, that is $X$ is a maximal $\theta$-nice set in $G$.
\end {itemize}
\end {thm}

In view of Theorem \ref{max-Tutte}, it suffices to show that (b) and (c) of Theorem \ref{theta-max-Tutte} are equivalent. Using the fact that $X$ is an $\theta$-extreme set, it is not hard to prove the following proposition.

\begin {prop}\label{one_side} Let $G$ be a graph and $X\subseteq V(G)$ with $\vert X\vert>1$. If $X$ is an $\theta$-extreme set then $X$ is $\theta$-nice.
\end {prop}

To complete the proof of Theorem \ref{theta-max-Tutte}, our aim for the rest of this section is to show that a $\theta$-nice set must be $\theta$-extreme. We shall need the following results.

\begin{lm}\label{path_interlacing}\textnormal {\cite[Corollary 2.5]{G}} For any root $\theta$ of $\mu(G,x)$ and a path $P$ in $G$,
\[ \textnormal{mult}  (\theta, G \setminus P) \ge \m(\theta, G)-1. \]
\end{lm}

\begin{lm}\label{godsil_positive}\textnormal {\cite[Theorem 4.2]{G}}
Let $u$ be a $\theta$-positive vertex in $G$. Then
\begin{itemize}
\item[\textnormal{(a)}] if $v$ is $\theta$-essential in $G$ then it
is $\theta$-essential in $G \setminus u$,
\item[\textnormal{(b)}] if $v$ is $\theta$-positive in $G$ then it
is $\theta$-essential or $\theta$-positive in $G \setminus u$,
\item[\textnormal{(c)}] if $u$ is $\theta$-neutral in $G$ then it is
$\theta$-essential or $\theta$-neutral in $G \setminus u$.
\end{itemize}
\end{lm}

\begin{rem}
The assertions of Lemma \ref{godsil_positive}, excluding part (b), still hold even if $\theta$ is not a root of $\mu(G,x)$.
\end{rem}

The following corollary is an immediate consequence of part (a) of Lemma \ref{godsil_positive}.
\begin{cor}\label{delete-essential-positive}
Suppose $u$ is $\theta$-positive and $v$ is $\theta$-essential in $G$. Then $u$ remains $\theta$-positive in $G \setminus v$.
\end{cor}

\begin{lm}\label{all_positive}  Let $u_1,u_2,\dots, u_k$ be $\theta$-positive vertices in $G$. Then $\textnormal{mult} (\theta, G\setminus u_1u_2\dots u_k)$ is either equal to $\textnormal{mult} (\theta, G)+k$ or at most $\textnormal{mult} (\theta, G)+k-2$.
\end{lm}

\begin{proof} We shall prove by induction on  $k$. Clearly it is true for $k=1$. Suppose $k\geq 2$. Assume that it is true for $k-1$, that is to say, $\textnormal{mult} (\theta, G\setminus u_1u_2\dots u_{k-1})$ is either equal to $\textnormal{mult} (\theta, G)+k-1$ or at most $\textnormal{mult} (\theta, G)+k-3$. In the latter we are done by Lemma \ref{interlacing}. In the former, $u_{i}$ is $\theta$-positive in $G \setminus u_{1} \cdots u_{i-1}$ for all $1 \le i \le k-1$. By Lemma \ref{godsil_positive}, $u_{i}$ is either $\theta$-positive or $\theta$-essential in $G \setminus u_{1} \cdots u_{i-1}$ for all $1 \le i \le k$. In particular, $u_{k}$ is either $\theta$-positive or $\theta$-essential in $G \setminus u_{1} \cdots u_{k-1}$ whence $\textnormal{mult} (\theta, G \setminus u_{1} \cdots u_{k}) = \textnormal{mult} (\theta, G)+k$ or $\textnormal{mult} (\theta, G)+k-2$.
\end{proof}

\begin{lm}\label{neutral_neighbor} Suppose $\theta \not = 0$ and $u$ is a $\theta$-essential vertex in $G$. Then $u$ has a neighbor which is $\theta$-neutral in $G \setminus u$.
\end{lm}

\begin{proof} By Lemma \ref{path_interlacing}, no neighbor of $u$ can be $\theta$-essential in $G \setminus u$. Suppose all neighbors of $u$ are $\theta$-positive in $G \setminus u$. Then, by comparing multiplicities of $\theta$ on both sides of the recurrence $\mu (G,x)=x\mu  (  G\setminus u,x)-\sum_{v\sim u} \mu  (  G\setminus uv,x)$ (part (c) of Theorem \ref{basic_property}) and the fact that $\theta \not = 0$, we observe that $\textnormal{mult} (\theta, G \setminus u) \ge \textnormal{mult} (\theta, G)$, contradicting the assumption that $u$ is $\theta$-essential in $G$.
\end{proof}

\begin{lm}\label{base_nice_case}  Let $G$ be a graph and $X\subseteq V(G)$ with $\vert X\vert=3$. If  $X$ is $\theta$-nice then $X$ is a $\theta$-extreme set.
\end{lm}

\begin{proof} The case $\theta=0$ is covered in Theorem \ref{extreme_zero}. So we may assume $\theta\neq 0$.  Let $X=\{x_1,x_2,x_3\}$ and $\textnormal {mult} (\theta,G)=k$ (We allow $k$ to take zero value). Now $\textnormal {mult} (\theta,G\setminus x_2)=k+1$ and $\textnormal {mult} (\theta,G\setminus x_2x_3)=k+2=\textnormal {mult} (\theta,G\setminus x_2x_1)$. This implies that $x_1$ and $x_3$ are $\theta$-positive in $G\setminus x_2$. By Lemma \ref{godsil_positive}, $x_1$ is either $\theta$-positive or $\theta$-essential in $G\setminus x_2x_3$.  If the former holds, then $\textnormal {mult} (\theta,G\setminus x_2x_3x_1)=k+3$ and $X$ is an $\theta$-extreme set.
So we may assume the latter holds. Then $\textnormal {mult} (\theta,G\setminus x_2x_3x_1)=k+1$.

By Lemma \ref{neutral_neighbor}, $x_1$ is adjacent to a vertex $z$ in $G\setminus x_2x_3$, where $z$ is $\theta$-neutral in $G\setminus x_2x_3x_1$. Therefore $\textnormal {mult} (\theta,G\setminus x_2x_3x_1z)=k+1$. By part (b) of Theorem \ref{basic_property}, $\mu (  G\setminus x_2x_3,x)=\mu (  (G\setminus x_2x_3)-e,x)-\mu (  G\setminus x_2x_3x_1z, x)$ where $e=(x_1,z)$ is an edge of $G$. Since $\textnormal {mult} (\theta,G\setminus x_2x_3)=k+2$, we must have $\textnormal {mult} (\theta,(G\setminus x_2x_3)-e)=k+1$.

\noindent
{\bf Case 1.} Suppose $z$ is $\theta$-essential in $G\setminus x_1$. Then $\textnormal{mult} (\theta, G\setminus x_1z)=k$.  Recall that $\textnormal{mult} (\theta, G\setminus x_1x_2)=k+2=\textnormal{mult} (\theta, G\setminus x_1x_3)$. This implies that $x_2$ and $x_3$ are $\theta$-positive in $G\setminus x_1$. By Corollary \ref{delete-essential-positive}, $x_2$ and $x_3$ are $\theta$-positive in $G\setminus x_1z$. By Lemma \ref{all_positive}, $\textnormal {mult} (\theta,G\setminus x_1zx_2x_3)$ is either equal to $\textnormal{mult} (\theta, G\setminus x_1z)+2=k+2$ or at most $k$, a contrary to the fact that $\textnormal {mult} (\theta,G\setminus x_2x_3x_1z)=k+1$.

\noindent
{\bf Case 2.} Suppose $z$ is $\theta$-neutral or $\theta$-positive in $G\setminus x_1$. Then $\textnormal{mult} (\theta, G\setminus x_1z)\geq k+1$.  By part (b) of Theorem \ref{basic_property}, $\mu (  G)=\mu (  G-e,x)-\mu (  G\setminus x_1z, x)$. By comparing the multiplicity of $\theta$ as zero on both sides of the equation, we deduce that $\textnormal{mult} (\theta, G-e)=k$. Note that $(G-e)\setminus x_1x_2=G\setminus x_1x_2$ and $(G-e)\setminus x_1x_3=G\setminus x_1x_3$. Therefore $\textnormal{mult} (\theta, (G-e)\setminus x_1x_2)=k+2=\textnormal{mult} (\theta, (G-e)\setminus x_1x_3)$. This implies that $x_2$ and $x_3$ are $\theta$-positive in $G-e$. By Lemma \ref{all_positive}, $\textnormal{mult} (\theta, (G-e)\setminus x_2x_3)$ is either equal to $k+2$ or at most $k$, a contrary to the fact that $\textnormal {mult} (\theta,(G\setminus x_2x_3)-e)=k+1$.

Hence $\textnormal {mult} (\theta,G\setminus x_2x_3x_1)=k+3$ and $X$ is an $\theta$-extreme set.
\end{proof}

\begin{thm}\label{general_nice_case}  Let $G$ be a graph and $X\subseteq V(G)$ with $\vert X\vert>1$. Then  $X$ is an $\theta$-extreme set if and only if  $X$ is $\theta$-nice.
\end{thm}

\begin{proof} By Proposition \ref{one_side}, it is sufficient to prove that if $X$ is $\theta$-nice then $X$ is an $\theta$-extreme set. We shall prove by induction on $\vert X\vert$. Clearly it is true when $\vert X\vert=2$. Let $\vert X\vert\geq 3$. Assume that it is true for all $\theta$-nice sets $X'$ with $\vert X'\vert<\vert X\vert$.

Let $a,b,c\in X$ and $X_1=X\setminus \{a,b,c\}$ ($X_1$ could be empty). Note that $X_1\cup \{a,b\}$, $X_1\cup \{a,c\}$ and $X_1\cup \{b,c\}$ are all $\theta$-nice sets. By induction, all of them are $\theta$-extreme sets.  By using Lemma \ref{interlacing}, it is not hard to deduce that $\textnormal{mult} (\theta, G\setminus X_1)=\textnormal{mult} (\theta, G)+\vert X_1\vert$. Then $\{a,b,c\}$ is a $\theta$-nice set in $G\setminus X_1$. By Lemma \ref{base_nice_case}, $\{a,b,c\}$ is an $\theta$-extreme set in $G\setminus X_1$ and $\textnormal{mult} (\theta, G\setminus X)=\textnormal{mult} (\theta, G)+\vert X_1\vert+3=\textnormal{mult} (\theta, G)+\vert X\vert$. Hence $X$ is an $\theta$-extreme set.
\end{proof}

\section{The $S_{\theta}$-graphs}

This section is devoted to the graph $S_{\theta}(G)$ which is a supergraph of $G$ obtained by joining any $\theta$-special vertex to all the other vertices. Formally,

\begin {dfn}\label {BP:D3b}  Let $G$ be a graph and $\theta$ be a real number. Then the graph $S_{\theta}(G)$ is defined by $V(S_{\theta}(G))=V(G)$ and $(w,z)\in E(S_{\theta}(G))$ if and only if $(w,z)\in E(G)$ or $w\in A_{\theta}(G)$ and $z\in V(G)$.
\end {dfn}

We shall prove that the graph $S_{\theta}(G)$ and $G$ have the same Gallai-Edmonds decomposition (Corollary \ref{similar_gallai_edmond}). We require the following lemmas:

\begin{lm}\label{existence_essential}\textnormal {\cite[Lemma 3.1]{G}} Suppose $\m(\theta, G)>0$. Then $G$ contains at least one $\theta$-essential vertex.
\end{lm}

\begin{lm}\label{Ku-Chen-neutral}\textnormal {\cite[Proposition 2.9]{KC}}
Let $u$ be a $\theta$-neutral vertex in $G$. Then
\begin{itemize}
\item[\textnormal{(a)}] if $v$ is $\theta$-positive in $G$ then it is
$\theta$-positive or $\theta$-neutral in $G \setminus u$;
\item[\textnormal{(b)}] if $v$ is $\theta$-essential in $G$ then it is
$\theta$-essential in $G \setminus u$;
\item[\textnormal{(c)}] if $v$ is $\theta$-neutral in $G$ then it is
$\theta$-neutral or $\theta$-positive in $G \setminus u$.
\end{itemize}
\end{lm}

\begin {thm}\label {BP:T3a}  Let $G$ be a graph. Let $u\in A_{\theta}(G)$ and $v\in V(G)$ where $(u,v)\notin E(G)$. Let $G'$ be the graph with $V(G')=V(G)$ and $E(G')=E(G)\cup \{(u,v)\}$. Then $\textnormal {mult} (\theta,G')=\textnormal {mult} (\theta,G)$ and
\begin {itemize}
\item [(a)] $B_{\theta}(G')=B_{\theta}(G)$,
\item [(b)] $P_{\theta}(G')=P_{\theta}(G)$,
\item [(c)] $N_{\theta}(G')=N_{\theta}(G)$,
\item [(d)] $A_{\theta}(G')=A_{\theta}(G)$.
\end {itemize}
\end {thm}

\begin {proof} Let $\textnormal {mult} (\theta,G)=k$. Then $\textnormal {mult} (\theta,G\setminus u)=k+1$. Also by part (b) of Theorem \ref {basic_property},
\begin{equation}
\mu (G',x)=\mu (G,x)-\mu (G\setminus uv,x),\tag {1}
\end{equation}

\noindent
{\bf Case 1.} Suppose $v\in B_{\theta}(G)$. Then by part (i) of Theorem \ref {P:T5}, $\textnormal {mult} (\theta,G\setminus uv)=k$.
We first show that $\textnormal {mult} (\theta,G')=\textnormal {mult} (\theta,G)$.  By comparing the multiplicity of $\theta$ as zero on both sides of the equation in (1), we deduce that $\textnormal {mult} (\theta,G')\geq k$. Note that $\mu (G'\setminus v,x)=\mu (G\setminus v,x)$. So $\textnormal {mult} (\theta,G'\setminus v)=k-1$. By Lemma \ref {interlacing}, $\textnormal {mult} (\theta,G')=k$.

Now we show that $u\in A_{\theta}(G')$. Since $\textnormal {mult} (\theta,G'\setminus v)=k-1$, $v\in B_{\theta}(G')$. On the other hand,   $\mu (G'\setminus u,x)=\mu (G\setminus u,x)$. So $\textnormal {mult} (\theta,G'\setminus u)=k+1$ and $u\in A_{\theta}(G')$, for $u$ is adjacent to $v$. By part (i) of Theorem \ref {P:T5}, we have $B_{\theta}(G')=B_{\theta}(G'\setminus u)=B_{\theta}(G\setminus u)=B_{\theta}(G)$. The proof of part (a) is complete. Now part (b), (c) and (d) follow easily from part (ii), (iii) and (iv) of Theorem \ref {P:T5}.

\noindent
{\bf Case 2.} Suppose $v\in N_{\theta}(G)$. Then by part (iii) of Theorem \ref {P:T5}, $\textnormal {mult} (\theta,G\setminus uv)=k+1$. Using (1) again, we deduce that  $\textnormal {mult} (\theta,G')=k=\textnormal {mult} (\theta,G)$.

Now we show that $u\in A_{\theta}(G')$. Since $u\in A_{\theta}(G)$, $u$ is adjacent to a $\theta$-essential vertex $w$ in $G$.  By part (i) of of Theorem \ref {P:T5}, $w\in B_{\theta}(G\setminus u)$. Recall that $v\in N_{\theta}(G\setminus u)$. So, by part (a) of Lemma \ref{Ku-Chen-neutral}, $w\in B_{\theta}(G\setminus uv)$. Therefore  $\textnormal {mult} (\theta,G\setminus uvw)=k$. Since $\textnormal {mult} (\theta,G\setminus w)=k-1$, we deduce from $\mu (G'\setminus w,x)=\mu (G\setminus w,x)-\mu (G\setminus uvw,x)$ (part (b) of Theorem \ref {basic_property}) that $\textnormal {mult} (\theta,G'\setminus w)=k-1$. Hence $u\in A_{\theta}(G')$. As before part (a), (b), (c) and (d) follow easily from part (i), (ii), (iii) and (iv) of Theorem \ref {P:T5}.

\noindent
{\bf Case 3.} The case when $v\in A_{\theta}(G)\cup P_{\theta}(G)$ is proved similarly.
\end {proof}

Note that when $A_{\theta}(G)=\varnothing$, $S_{\theta}(G)=G$. Now by repeatedly applying Theorem \ref {BP:T3a}, we have the following corollary.

\begin {cor}\label{similar_gallai_edmond}  Let $G$ be a graph.  Then $\textnormal {mult} (\theta,S_{\theta}(G))=\textnormal {mult} (\theta,G)$
\begin {itemize}
\item [(a)] $B_{\theta}(S_{\theta}(G))=B_{\theta}(G)$,
\item [(b)] $P_{\theta}(S_{\theta}(G))=P_{\theta}(G)$,
\item [(c)] $N_{\theta}(S_{\theta}(G))=N_{\theta}(G)$,
\item [(d)] $A_{\theta}(S_{\theta}(G))=A_{\theta}(G)$.
\end {itemize}
\end {cor}

Two graphs $G$ and $G'$ are said to have the \emph{same Gallai-Edmonds decomposition} with respect to $\theta$, if there is a bijection, $\psi :V(G)\rightarrow V(G')$ such that $\psi(A_{\theta}(G))=A_{\theta}(G')$ and the restriction of $\psi$ to $G\setminus A_{\theta}(G)$ is an isomorphism onto $G'\setminus A_{\theta}(G')$.

Corollary \ref{similar_gallai_edmond} asserts that the Gallai-Edmonds decomposition of $G$ is stable under the $S_{\theta}$-operator. Since $G\setminus A_{\theta}(G)=S_{\theta}(G)\setminus A_{\theta}(S_{\theta}(G))$, we conclude that $G$ and $S_{\theta}(G)$ have the same Gallai-Edmonds decomposition with respect to $\theta$ and this proves Theorem \ref{Main03}. Corollary \ref{similar_gallai_edmond} also allows us, for the rest of this section, to predict the multiplicity of $\theta$ upon deleting two vertices of $S_{\theta}(G)$ in terms of the Gallai-Edmonds decomposition (see Corollary \ref{S:C6} below).

\begin {lm}\label {S:L4a}  Let $G$ be a graph and $S\subseteq V(G)$ be a set for which each $s\in S$ is adjacent to every other vertices in $G$. Suppose $\textnormal {mult} (\theta, G\setminus v)\geq 1$ with $v\in V(G)\setminus S$ and $s\notin B_{\theta}(G\setminus v)$ for all $s\in S$.
Then $S$ is an $\theta$-extreme set in $G\setminus uv$ for all $u\in V(G)\setminus S$, $u\neq v$.
\end {lm}

\begin {proof} If $S=\varnothing$, we are done. Suppose $S\neq \varnothing$. By Lemma \ref{existence_essential}, $B_{\theta}(G\setminus v)\neq \varnothing$. Since  $s$ is adjacent to every other vertices in $G$ and  $s\notin B_{\theta}(G\setminus v)$, $s\in A_{\theta}(G\setminus v)$.
Hence $S\subseteq A_{\theta}(G\setminus v)$ and by part (iv) of Theorem \ref {P:T5},
\begin {equation}
\textnormal {mult} (\theta, (G\setminus v)\setminus S)=\textnormal {mult} (\theta, G\setminus v)+\vert S\vert.\notag
\end {equation}

Suppose $u\in B_{\theta}(G\setminus v)$. Then $\textnormal {mult} (\theta, G\setminus uv)=\textnormal {mult} (\theta, G\setminus v)-1$.  By part (i) of Theorem \ref {P:T5}, $u\in B_{\theta}((G\setminus v)\setminus S)$. So $\textnormal {mult} (\theta, (G\setminus uv)\setminus S)=\textnormal {mult} (\theta, G\setminus v)+\vert S\vert-1=\textnormal {mult} (\theta, G\setminus uv)+\vert S\vert$ and $S$ is an $\theta$-extreme set in $G\setminus uv$.

The case $u\in A_{\theta}(G\setminus v)\cup N_{\theta}(G\setminus v)\cup P_{\theta}(G\setminus v)$ is proved similarly.
\end {proof}

\begin {thm}\label{extreme_set_in_S}  Let $G$ be a graph. Let $H_1,\dots, H_q,Q_1,\dots, Q_m$ be all the components in $G\setminus A_{\theta}(G)$ with $H_i$ is $\theta$-critical for all $i$ and $\textnormal {mult} (\theta, Q_j)=0$ for all $j$.
 Suppose
\begin {itemize}
\item [(a)] $u\in V(G)\setminus A_{\theta}(G)$ and $v\in V(Q_{j_0})$ for some $j_0$, or
\item [(b)] $u,v\in V(H_{i_0})$ for some $i_0$, or
\item [(c)]  $\textnormal {mult} (\theta, G)\geq 2$, $u\in V(H_{i_1})$ and $v\in V(H_{i_2})$ for some $i_1,i_2$ and $i_1\neq i_2$.
\end {itemize}
Then $A_{\theta}(G)$ is an $\theta$-extreme set in $S_{\theta}(G)\setminus uv$.
\end {thm}

\begin {proof} By Corollary \ref {similar_gallai_edmond}, $A_{\theta}(G)=A_{\theta}(S_{\theta}(G))$. If $A_{\theta}(S_{\theta}(G))=\varnothing$, we are done. So we may assume $A_{\theta}(S_{\theta}(G))\neq \varnothing$. This also means that $\textnormal {mult} (\theta, S_{\theta}(G))\geq 1$.

Note that $S_{\theta}(G)\setminus A_{\theta}(S_{\theta}(G))=G\setminus A_{\theta}(G)$. So $H_1,\dots, H_q,Q_1,\dots, Q_m$ are all the components in $S_{\theta}(G)\setminus A_{\theta}(S_{\theta}(G))$ with $H_i$ is $\theta$-critical for all $i$ and $\textnormal {mult} (\theta, Q_j)=0$ for all $j$.

By Lemma \ref  {S:L4a}, it is sufficient to show that  $\textnormal {mult} (\theta, S_{\theta}(G)\setminus v)\geq 1$ and $w\notin B_{\theta}(S_{\theta}(G)\setminus v)$ for all $w\in A_{\theta}(S_{\theta}(G))$.

\noindent
(a) By  Theorem \ref {P:T5}, $v\in N_{\theta}(G)\cup P_{\theta}(G)$, and by Corollary \ref {similar_gallai_edmond}, $v\in N_{\theta}(S_{\theta}(G))\cup P_{\theta}(S_{\theta}(G))$. Therefore $\textnormal {mult} (\theta, S_{\theta}(G)\setminus v)\geq 1$. Let $w\in A_{\theta}(S_{\theta}(G))$. By Theorem \ref {P:T5}, $v\in N_{\theta}(S_{\theta}(G)\setminus w)\cup P_{\theta}(S_{\theta}(G)\setminus w)$. Therefore $\textnormal {mult} (\theta, S_{\theta}(G)\setminus wv)\geq \textnormal {mult} (\theta, S_{\theta}(G)\setminus w)\geq \textnormal {mult} (\theta, S_{\theta}(G)\setminus v)$. This implies that $w\notin B_{\theta}(S_{\theta}(G)\setminus v)$. Hence $w\notin B_{\theta}(S_{\theta}(G)\setminus v)$ for all $w\in A_{\theta}(S_{\theta}(G))$.

\noindent
(b) and (c). Suppose $\textnormal {mult} (\theta, G)\geq 2$ and $v\in V(H_{i_0})$ for some $i_0$. By  Theorem \ref {P:T5} and Corollary \ref {similar_gallai_edmond}, $v\in B_{\theta}(S_{\theta}(G))$ and $\textnormal {mult} (\theta, S_{\theta}(G))\geq 2$. Therefore $\textnormal {mult} (\theta, S_{\theta}(G)\setminus v)=\textnormal {mult} (\theta, S_{\theta}(G))-1\geq 1$.
Let $w\in A_{\theta}(S_{\theta}(G))$. By Theorem \ref {P:T5}, $v\in B_{\theta}(S_{\theta}(G)\setminus w)$. So $\textnormal {mult} (\theta, S_{\theta}(G)\setminus wv)=\textnormal {mult} (\theta, S_{\theta}(G)\setminus w)-1= \textnormal {mult} (\theta, S_{\theta}(G))>\textnormal {mult} (\theta, S_{\theta}(G)\setminus v)$. This implies that $w\notin B_{\theta}(S_{\theta}(G)\setminus u)$. Hence $w\notin B_{\theta}(S_{\theta}(G)\setminus v)$ for all $w\in A_{\theta}(S_{\theta}(G))$.

It is left only to show case (b) with  $\textnormal {mult} (\theta, G)=1$ and $v\in V(H_{i_0})$ for some $i_0$. Note that
$v\in B_{\theta}(S_{\theta}(G))$ and $\textnormal {mult} (\theta, S_{\theta}(G))=1$. Now $\textnormal {mult} (\theta, S_{\theta}(G)\setminus v)=0$. By Lemma \ref {interlacing}, $\textnormal {mult} (\theta, S_{\theta}(G)\setminus vu)=0$ or 1.

Suppose  $\textnormal {mult} (\theta, S_{\theta}(G)\setminus vu)=0$. By Lemma \ref {interlacing} again, $\textnormal {mult} (\theta, (S_{\theta}(G)\setminus vu)\setminus  A_{\theta}(S_{\theta}(G)))\leq \vert A_{\theta}(S_{\theta}(G))\vert$. On the other hand, by part (a) of Theorem \ref {basic_property} and Corollary \ref {P:C7},
\begin {equation}
\textnormal {mult} (\theta, (S_{\theta}(G)\setminus  A_{\theta}(S_{\theta}(G)))\setminus uv)=\textnormal {mult} (\theta, H_{i_0}\setminus uv)+q-1=\textnormal {mult} (\theta, H_{i_0}\setminus uv)+\vert A_{\theta}(S_{\theta}(G))\vert.\tag {2}
\end {equation}
Therefore $\textnormal {mult} (\theta, H_{i_0}\setminus uv)+\vert A_{\theta}(S_{\theta}(G))\vert\leq \vert A_{\theta}(S_{\theta}(G))\vert$ and $\textnormal {mult} (\theta, H_{i_0}\setminus uv)=0$. Hence $A_{\theta}(G)$ is an $\theta$-extreme set in $S_{\theta}(G)\setminus uv$.

Suppose  $\textnormal {mult} (\theta, S_{\theta}(G)\setminus vu)=1$. Let $w\in A_{\theta}(S_{\theta}(G))$. By Lemma \ref {interlacing},
\begin {equation}
\textnormal {mult} (\theta, (S_{\theta}(G)\setminus vuw)\setminus  (A_{\theta}(S_{\theta}(G))\setminus w))\leq \textnormal {mult} (\theta, S_{\theta}(G)\setminus vuw)+\vert A_{\theta}(S_{\theta}(G))\vert -1.\notag
\end {equation}

 On the other hand, (2) holds.
Therefore $\textnormal {mult} (\theta, H_{i_0}\setminus uv)+\vert A_{\theta}(S_{\theta}(G))\vert\leq \textnormal {mult} (\theta, S_{\theta}(G)\setminus vuw)+\vert A_{\theta}(S_{\theta}(G))\vert -1$ and $\textnormal {mult} (\theta, S_{\theta}(G)\setminus vuw)\geq 1=\textnormal {mult} (\theta, S_{\theta}(G)\setminus vu)$. So $w\notin B_{\theta}(S_{\theta}(G)\setminus uv)$. Since $w$ is adjacent to every other vertices in $S_{\theta}(G)$, $w\in A_{\theta}(S_{\theta}(G)\setminus uv)$. Hence $A_{\theta}(S_{\theta}(G))\subseteq A_{\theta}(S_{\theta}(G)\setminus uv)$ and  $A_{\theta}(G)$ is an $\theta$-extreme set in $S_{\theta}(G)\setminus uv$.
\end {proof}

\begin {cor}\label {S:C6}  Let $G$ be a graph. Let $H_1,\dots, H_q,Q_1,\dots, Q_m$ be all the components in $G\setminus A_{\theta}(G)$ with $H_i$ is $\theta$-critical for all $i$ and $\textnormal {mult} (\theta, Q_j)=0$ for all $j$.
Then the following holds:
\begin {itemize}
\item [(a)] If $u\in V(H_{i_0})$ and $v\in V(Q_{j_0})$ for some $i_0,j_0$, then
\begin {equation}
\textnormal {mult} (\theta, S_{\theta}(G)\setminus uv)=\textnormal {mult} (\theta, S_{\theta}(G))-1+\textnormal {mult} (\theta, Q_{j_0}\setminus v).\notag
\end {equation}
\item [(b)] If $u,v\in V(Q_{j_0})$ for some $j_0$, then
\begin {equation}
\textnormal {mult} (\theta, S_{\theta}(G)\setminus uv)=\textnormal {mult} (\theta, S_{\theta}(G))+\textnormal {mult} (\theta, Q_{j_0}\setminus uv).\notag
\end {equation}
\item [(c)] If $u\in V(Q_{j_1})$ and $v\in V(Q_{j_2})$ for some $j_1,j_2$, $j_1\neq j_2$, then
\begin {equation}
\textnormal {mult} (\theta, S_{\theta}(G)\setminus uv)=\textnormal {mult} (\theta, S_{\theta}(G))+\textnormal {mult} (\theta, Q_{j_1}\setminus u)+\textnormal {mult} (\theta, Q_{j_2}\setminus v).\notag
\end {equation}
\item [(d)] If $u,v\in V(H_{i_0})$ for some $i_0$, then
\begin {equation}
\textnormal {mult} (\theta, S_{\theta}(G)\setminus uv)=\textnormal {mult} (\theta, S_{\theta}(G))-1+\textnormal {mult} (\theta, H_{i_0}\setminus uv).\notag
\end {equation}
\item [(e)]  If $\textnormal {mult} (\theta, G)\geq 2$, $u\in V(H_{i_1})$ and $v\in V(H_{i_2})$ for some $i_1,i_2$, $i_1\neq i_2$, then
\begin {equation}
\textnormal {mult} (\theta, S_{\theta}(G)\setminus uv)=\textnormal {mult} (\theta, S_{\theta}(G))-2.\notag
\end {equation}
\end {itemize}
\end {cor}

\begin {proof} (a) Suppose $u\in V(H_{i_0})$ and $v\in V(Q_{j_0})$ for some $i_0,j_0$. By  Theorem \ref {extreme_set_in_S}, $A_{\theta}(G)$ is an $\theta$-extreme set in $S_{\theta}(G)\setminus uv$. Therefore
\begin {equation}
\textnormal {mult} (\theta, S_{\theta}(G)\setminus (A_{\theta}(S_{\theta}(G))\cup \{u,v\}))=\textnormal {mult} (\theta, S_{\theta}(G)\setminus uv)+\vert A_{\theta}(G)\vert.\notag
\end {equation}

 Recall that $S_{\theta}(G)\setminus A_{\theta}(S_{\theta}(G))=G\setminus A_{\theta}(G)$. By Corollary \ref {P:C7}, and part (a) of Theorem \ref {basic_property}, we have
 \begin {align}
\textnormal {mult} (\theta, G\setminus (A_{\theta}(G)\cup \{u,v\})) &=\textnormal {mult} (\theta, Q_{j_0}\setminus v)+\textnormal {mult} (\theta, H_{i_0}\setminus u)+\sum_{1\leq i\leq q, i\neq i_0} \textnormal {mult} (\theta, H_i)\notag\\
&=\textnormal {mult} (\theta, Q_{j_0}\setminus v)+\textnormal {mult} (\theta, G)+\vert A_{\theta}(G)\vert-1.\notag
\end {align}
 This implies that $\textnormal {mult} (\theta, S_{\theta}(G)\setminus uv)=\textnormal {mult} (\theta, Q_{j_0}\setminus v)+\textnormal {mult} (\theta, G)-1=\textnormal {mult} (\theta, Q_{j_0}\setminus v)+\textnormal {mult} (\theta, S_{\theta}(G))-1$, where the last inequality follows from Corollary \ref{similar_gallai_edmond}.

(b), (c), (d) and (e) are proved similarly.
\end {proof}

\section{The graphs $D_{\theta}(G)$ and $D_{\theta}(S_{\theta}(G))$}

In this section, we shall determine the edge-set of $D_{\theta}(G)$ in terms of its Gallai-Edmonds decomposition (Theorem \ref{d_graph_for_G}). Finally, we shall prove that $D_{\theta}(G) = D_{\theta}(S_{\theta}(G))$ (Corollary \ref{gallai_Edmond_decomposition_G_S}).

First, we list all possibilities for $\m(\theta, G \setminus uv)$ with respect its Gallai-Edmonds decomposition:

\begin {lm}\label {BP:L4a} Let $G$ be a graph. Then the following hold.
\begin {itemize}
\item [(a)] If $u\in B_{\theta}(G)$ then $\textnormal {mult} (\theta, G)-2\leq \textnormal {mult} (\theta, G\setminus uv)\leq \textnormal {mult} (\theta, G)$ for all $v\in V(G)\setminus \{u\}$.
\item [(b)] If $u\in P_{\theta}(G)$ then $\textnormal {mult} (\theta, G)\leq \textnormal {mult} (\theta, G\setminus uv)\leq \textnormal {mult} (\theta, G)+2$ for all $v\in V(G)\setminus \{u\}$.
\item [(c)] If $u\in N_{\theta}(G)$ then $\textnormal {mult} (\theta, G)-1\leq \textnormal {mult} (\theta, G\setminus uv)\leq \textnormal {mult} (\theta, G)+1$ for all $v\in V(G)\setminus \{u\}$.
\item [(d)] If $u\in A_{\theta}(G)$ then
\begin {itemize}
\item [(i)]  $\textnormal {mult} (\theta, G\setminus uv)=\textnormal {mult} (\theta, G)+1$ whenever $v\in N_{\theta}(G)$,
\item [(ii)] $\textnormal {mult} (\theta, G\setminus uv)=\textnormal {mult} (\theta, G)+2$ whenever $v\in P_{\theta}(G)\cup (A_{\theta}(G)\setminus \{u\})$,
\item [(iii)] $\textnormal {mult} (\theta, G\setminus uv)=\textnormal {mult} (\theta, G)$ whenever $v\in B_{\theta}(G)$.
\end {itemize}
\end {itemize}
\end {lm}

\begin {proof} Clearly, if $u \in B_{\theta}(G)$, then $\textnormal {mult} (\theta, G\setminus u)=\textnormal {mult} (\theta, G)-1$. So part (a) follows from Lemma \ref {interlacing}. Part (b) and (c) are proved similarly. Part (d) follows from Theorem \ref {P:T5}.
\end {proof}

Recall that $D_{\theta}(G) = D_{-2, \theta}(G) \cup D_{-1,\theta}(G) \cup D_{0,\theta}(G)$. Therefore, in order to determine the edges in $D_{\theta}(G)$, we can first determine the edges in $D_{r, \theta}(G)$ for $r=-2,-1,0$. However the graphs $D_{r,\theta}(G)$ do not behave `nicely'. Therefore we shall study $D_{r,\theta}(S_{\theta}(G))$ instead.  In fact, we shall do this for all $r=-2,-1,0,1,2$.

\subsection{$D_{-2,\theta}(G)$}

\begin {lm}\label {BP:L4} Let $G$ be a graph with $\textnormal {mult} (\theta, G)=0$ or 1. Then $D_{-2,\theta}(G)$ is an empty graph with $\vert V(G)\vert$ vertices.
\end {lm}

\begin {proof} Since $\textnormal {mult} (\theta, G\setminus uv)\geq 0$ for all $u,v\in V(G)$,  we can never have $\textnormal {mult} (\theta, G\setminus uv)=\textnormal {mult} (\theta, G)-2$. Hence the lemma holds.
\end {proof}

\begin {lm}\label {BP:L6a} Let $G$ be a graph with $\textnormal {mult} (\theta, G)\geq 2$. Let $H_1,\dots, H_q$ be all the $\theta$-critical components in $G\setminus A_{\theta}(G)$. If $(u,v)\in E(D_{-2,\theta}(G))$, then $u\in V(H_i)$ and $v\in V(H_j)$ for some $i\neq j$.
\end {lm}

\begin {proof} Suppose $(u,v)\in E(D_{-2,\theta}(G))$. By Lemma \ref {BP:L4a}, we must have $u,v\in B_{\theta}(G)$. By part (iv) of Corollary \ref {P:C7}, $u,v\in V(H_1)\cup\cdots\cup V(H_q)$.  Suppose $u,v\in V(H_{j_0})$ for some $j_0$. By Corollary \ref {P:C7} and part (a) of Theorem \ref {basic_property}, we have
$\textnormal {mult} (\theta, G\setminus A_{\theta}(G))=\sum_{i=1}^q \textnormal {mult} (\theta, H_i)=q=\textnormal {mult} (\theta, G)+\vert A_{\theta}(G)\vert$. Note that $\textnormal {mult} (\theta, H_{j_0}\setminus u)=0$. Therefore $\textnormal {mult} (\theta, H_{j_0}\setminus uv)\geq 0$ and that $\textnormal {mult} (\theta, G\setminus (A_{\theta}(G)\cup \{u,v\}))=\textnormal {mult} (\theta, H_{j_0}\setminus uv)+\sum_{1\leq i\leq q, i\neq j_0} \textnormal {mult} (\theta, H_i)\geq q-1=\textnormal {mult} (\theta, G)+\vert A_{\theta}(G)\vert-1$.

On the other hand, $\textnormal {mult} (\theta, G\setminus uv)=\textnormal {mult} (\theta, G)-2$. By Lemma \ref {interlacing}, $\textnormal {mult} (\theta, G\setminus (A_{\theta}(G)\cup \{u,v\}))\leq \textnormal {mult} (\theta, G)-2+\vert A_{\theta}(G)\vert$, a contradiction. Hence $u\in V(H_i)$ and $v\in V(H_j)$ for some $i\neq j$.
\end {proof}

Note that in general the converse of Lemma \ref {BP:L6a} is not true. In the following graph $G$ (see Figure 1), we have $A_{1}(G)=\{u,v\}$ and $H_1,H_2,H_3,H_4$ are all the $1$-critical components in $G\setminus A_{1}(G)$. Now $\textnormal {mult} (1,G)=2$ and $w\in  V(H_1)$, $z\in V(H_2)$. But $\textnormal {mult} (1,G\setminus wz)=1\neq 0=\textnormal {mult} (1,G)-2$.
\begin{center}
\begin{pspicture}(0,0)(6,6)
\cnodeput(1.5, 5){1}{}
\cnodeput(5, 5){2}{}
\cnodeput(0.5, 4){3}{}
\cnodeput(0.5, 3){4}{}
\cnodeput(2.5, 4){5}{}
\cnodeput(2.5, 3){6}{}
\cnodeput(4, 4){7}{}
\cnodeput(4, 3){8}{}
\cnodeput(6, 4){9}{}
\cnodeput(6, 3){10}{}
\ncline{1}{3}
\ncline{1}{5}
\ncline{2}{5}
\ncline{2}{7}
\ncline{2}{9}
\ncline{3}{4}
\ncline{5}{6}
\ncline{7}{8}
\ncline{9}{10}
\psellipse[linestyle=dotted](0.5,3)(0.5,1.5)
\psellipse[linestyle=dotted](2.5,3)(0.5,1.5)
\psellipse[linestyle=dotted](4,3)(0.5,1.5)
\psellipse[linestyle=dotted](6,3)(0.5,1.5)

\rput(3,1){Figure 1.}
\rput(-1,4){$G=$}
\rput(0.5,2){$H_1$}
\rput(2.5,2){$H_2$}
\rput(4,2){$H_3$}
\rput(6,2){$H_4$}
\rput(1.8,5){$u$}
\rput(5.3,5){$v$}
\rput(0.8,3){$w$}
\rput(2.8,3){$z$}
\end{pspicture}
\end{center}
However it is true for the graph $S_{\theta}(G)$ (see Theorem \ref {d_graph_2_negative}).

\begin {thm}\label{d_graph_2_negative} Let $G$ be a graph with $\textnormal {mult} (\theta, G)\geq 2$. Let $H_1,\dots, H_q$ be all the $\theta$-critical components in $G\setminus A_{\theta}(G)$. Then  $(u,v)\in E(D_{-2,\theta}(S_{\theta}(G)))$ if and only if $u\in V(H_i)$ and $v\in V(H_j)$ for some $i\neq j$.
\end {thm}

\begin {proof} Suppose $u\in V(H_i)$ and $v\in V(H_j)$ for some $i\neq j$. By part (e) of Corollary \ref {S:C6}, we have  $\textnormal {mult} (\theta, S_{\theta}(G)\setminus uv)=\textnormal {mult} (\theta, S_{\theta}(G))-2$. So $(u,v)\in E(D_{-2,\theta}(S_{\theta}(G)))$.

The converse follows from  Lemma \ref {BP:L6a} (Recall that $S_{\theta}(G)\setminus A_{\theta}(S_{\theta}(G))=G\setminus A_{\theta}(G)$).
\end {proof}

\subsection{$D_{-1,\theta}(G)$}

The proof of the following lemma is similar to Lemma \ref {BP:L4} and therefore is omitted.
\begin {lm}\label {BP:L9} Let $G$ be a graph with $\textnormal {mult} (\theta, G)=0$. Then $D_{-1,\theta}(G)$ is an empty graph with $\vert V(G)\vert$ vertices.
\end {lm}

Using Lemma \ref{Ku-Chen-neutral} and Lemma \ref {BP:L4a}, one can easily deduce Lemma \ref {BP:L10}.
\begin {lm}\label {BP:L10} Let $G$ be a graph with $\textnormal {mult} (\theta, G)\geq 1$. If $(u,v)\in E(D_{-1,\theta}(G))$ then either $u\in N_{\theta}(G)$ and $v\in B_{\theta}(G)$ or $u,v\in B_{\theta}(G)$.
\end {lm}

\begin {thm}\label{d_graph_1_negative} Let $G$ be a graph with $\textnormal {mult} (\theta, G)\geq 2$. Let $H_1,\dots, H_q$ be all the $\theta$-critical components in $G\setminus A_{\theta}(G)$. Then  $(u,v)\in E(D_{-1,\theta}(S_{\theta}(G)))$ if and only if
\begin {itemize}
\item [(a)] $u\in N_{\theta}(G)$ and $v\in B_{\theta}(G)$, or
\item [(b)] $(u,v)\in E(D_{-1,\theta}(H_{i_0}))$ for some $i_0$.
\end {itemize}
\end {thm}

\begin {proof} Suppose (a) holds. By Corollary \ref {similar_gallai_edmond}, $u\in N_{\theta}(S_{\theta}(G))$ and $v\in B_{\theta}(S_{\theta}(G))$. By Lemma \ref{Ku-Chen-neutral}, $\textnormal {mult} (\theta, S_{\theta}(G)\setminus uv)=\textnormal {mult} (\theta, S_{\theta}(G))-1$. Thus $(u,v)\in E(D_{-1,\theta}(S_{\theta}(G)))$.

Suppose (b) holds. Then $\textnormal {mult} (\theta, H_{i_0}\setminus uv)=0$. By part (d) of Corollary \ref {S:C6}, $
\textnormal {mult} (\theta, S_{\theta}(G)\setminus uv)=\textnormal {mult} (\theta, S_{\theta}(G))-1+\textnormal {mult} (\theta, H_{i_0}\setminus uv)=\textnormal {mult} (\theta, S_{\theta}(G))-1$.  Hence  $(u,v)\in E(D_{-1,\theta}(S_{\theta}(G)))$.

Suppose $(u,v)\in E(D_{-1,\theta}(S_{\theta}(G)))$. By  Lemma \ref {BP:L10}, we may assume that $u,v\in B_{\theta}(G)$. Note that $H_1,\dots, H_q$ are all the $\theta$-critical components in $S_{\theta}(G)\setminus A_{\theta}(S_{\theta}(G))$. By part (d) and (e) of Corollary \ref {S:C6}, we must have $u,v\in  V(H_{i_0})$ for some $i_0$.
Therefore $\textnormal {mult} (\theta, S_{\theta}(G))-1=
\textnormal {mult} (\theta, S_{\theta}(G)\setminus uv)=\textnormal {mult} (\theta, S_{\theta}(G))-1+\textnormal {mult} (\theta, H_{i_0}\setminus uv)$, which implies that $\textnormal {mult} (\theta, H_{i_0}\setminus uv)=0$. Hence  $(u,v)\in E(D_{-1,\theta}(H_{i_0}))$.
\end {proof}

\subsection{$D_{0,\theta}(G)$}

Using Lemma \ref{godsil_positive}, and Lemma \ref {BP:L4a}, one can easily deduce Lemma \ref {BP:L13}.
\begin {lm}\label {BP:L13} Let $G$ be a graph. If $(u,v)\in E(D_{0,\theta}(G))$ then either $u\in P_{\theta}(G)\cup A_{\theta}(G)$ and $v\in B_{\theta}(G)$ or $u,v\in B_{\theta}(G)$ or $u,v\in P_{\theta}(G)\cup N_{\theta}(G)$.
\end {lm}

\begin {thm}\label{d_graph_0_neutral} Let $G$ be a graph with $\m(\theta, G)\ge 2$ and  $H_1,\dots, H_q,Q_1,\dots, Q_m$ be all the components in $G\setminus A_{\theta}(G)$ with $H_i$ is $\theta$-critical for all $i$ and $\textnormal {mult} (\theta, Q_j)=0$ for all $j$. Then  $(u,v)\in E(D_{0,\theta}(S_{\theta}(G)))$ if and only if
\begin {itemize}
\item [(a)] $u\in P_{\theta}(G)\cup A_{\theta}(G)$ and $v\in B_{\theta}(G)$, or
\item [(b)] $u,v\in N_{\theta}(G)$ with $u\in V(Q_{j_1})$ and $v\in V(Q_{j_2})$ for some $j_1$ and $j_2$, $j_1\neq j_2$, or
\item [(c)] $(u,v)\in E(D_{0,\theta}(H_{i_0}))$ for some $i_0$, or
\item [(d)] $(u,v)\in E(D_{0,\theta}(Q_{j_0}))$ for some $j_0$.
\end {itemize}
\end {thm}

\begin {proof} Suppose (a) holds. Then it follows from Lemma \ref{godsil_positive} that $(u,v)\in E(D_{0,\theta}(S_{\theta}(G)))$.

Suppose (b) holds. By   Theorem \ref {P:T5}, $u,v\in N_{\theta}(G\setminus A_{\theta}(G))$. By using part (a) of Theorem \ref {basic_property}, it is not hard to deduce that $\textnormal {mult} (\theta, Q_{j_1}\setminus u)=0=\textnormal {mult} (\theta, Q_{j_2}\setminus v)$. Then by part (c) of Corollary \ref {S:C6}, $\textnormal {mult} (\theta, S_{\theta}(G)\setminus uv)=\textnormal {mult} (\theta, S_{\theta}(G))+\textnormal {mult} (\theta, Q_{j_1}\setminus u)+\textnormal {mult} (\theta, Q_{j_2}\setminus v)=\textnormal {mult} (\theta, S_{\theta}(G))$. Hence $(u,v)\in E(D_{0,\theta}(S_{\theta}(G)))$.

Suppose (c) holds. Then $\textnormal {mult} (\theta, H_{i_0}\setminus uv)=1$. By part (d) of Corollary \ref {S:C6}, $
\textnormal {mult} (\theta, S_{\theta}(G)\setminus uv)=\textnormal {mult} (\theta, S_{\theta}(G))-1+\textnormal {mult} (\theta, H_{i_0}\setminus uv)=\textnormal {mult} (\theta, S_{\theta}(G))$.  Hence  $(u,v)\in E(D_{0,\theta}(S_{\theta}(G)))$.

Suppose (d) holds. Then $\textnormal {mult} (\theta, Q_{j_0}\setminus uv)=0$. By part (b) of Corollary \ref {S:C6}, $
\textnormal {mult} (\theta, S_{\theta}(G)\setminus uv)=\textnormal {mult} (\theta, S_{\theta}(G))+\textnormal {mult} (\theta, Q_{j_0}\setminus uv)=\textnormal {mult} (\theta, S_{\theta}(G))$.  Hence  $(u,v)\in E(D_{0,\theta}(S_{\theta}(G)))$.

Suppose $(u,v)\in E(D_{0,\theta}(S_{\theta}(G)))$. By  Lemma \ref {BP:L13}, we may assume that $u,v\in B_{\theta}(G)$ or $u,v\in P_{\theta}(G)\cup N_{\theta}(G)$. Suppose $u,v\in B_{\theta}(G)$.  By part (d) and (e) of Corollary \ref {S:C6}, we must have $u,v\in  V(H_{i_0})$ for some $i_0$.
So $\textnormal {mult} (\theta, S_{\theta}(G))=
\textnormal {mult} (\theta, S_{\theta}(G)\setminus uv)=\textnormal {mult} (\theta, S_{\theta}(G))-1+\textnormal {mult} (\theta, H_{i_0}\setminus uv)$, which implies that $\textnormal {mult} (\theta, H_{i_0}\setminus uv)=1$. Hence  $(u,v)\in E(D_{0,\theta}(H_{i_0}))$.

Suppose $u,v\in P_{\theta}(G)\cup N_{\theta}(G)$. If $u,v\in V(Q_{j_0})$ for some $j_0$, then by part (b) of Corollary \ref {S:C6}, we have $\textnormal {mult} (\theta, S_{\theta}(G))=\textnormal {mult} (\theta, S_{\theta}(G)\setminus uv)=\textnormal {mult} (\theta, S_{\theta}(G))+\textnormal {mult} (\theta, Q_{j_0}\setminus uv)$, which implies that $\textnormal {mult} (\theta, Q_{j_0}\setminus uv)=0$, i.e., $(u,v)\in E(D_{0,\theta}(Q_{j_0}))$.

If $u\in V(Q_{j_1})$ and $v\in V(Q_{j_2})$ for some $j_1,j_2$, $j_1\neq j_2$, then by part (c) of Corollary \ref {S:C6},
$\textnormal {mult} (\theta, S_{\theta}(G))=\textnormal {mult} (\theta, S_{\theta}(G)\setminus uv)=\textnormal {mult} (\theta, S_{\theta}(G))+\textnormal {mult} (\theta, Q_{j_1}\setminus u)+\textnormal {mult} (\theta, Q_{j_2}\setminus v)$, which implies  $\textnormal {mult} (\theta, Q_{j_1}\setminus u)=0=\textnormal {mult} (\theta, Q_{j_2}\setminus v)$. By using part (a) of Theorem \ref {basic_property}, we can deduce that $u,v\in N_{\theta}(G\setminus A_{\theta}(G))$. It then follows from Theorem \ref {P:T5}, that $u, v\in N_{\theta}(G)$.
\end {proof}

\subsection{$D_{1,\theta}(G)$}

Using  Lemma \ref {BP:L4a}, one can easily deduce Lemma \ref {BP:L16}.
\begin {lm}\label {BP:L16} Let $G$ be a graph. If $(u,v)\in E(D_{1,\theta}(G))$ then either $u\in A_{\theta}(G)$ and $v\in N_{\theta}(G)$ or $u,v\in P_{\theta}(G)\cup N_{\theta}(G)$.
\end {lm}

\begin {thm}\label{d_graph_1_positive} Let $G$ be a graph and  $Q_1,\dots, Q_m$ be all the components in $G\setminus A_{\theta}(G)$ with \linebreak $\textnormal {mult} (\theta, Q_j)=0$ for all $j$. Then  $(u,v)\in E(D_{1,\theta}(S_{\theta}(G)))$ if and only if
\begin {itemize}
\item [(a)] $u\in  A_{\theta}(G)$ and $v\in N_{\theta}(G)$, or
\item [(b)] $u\in P_{\theta}(G)$ and $v\in N_{\theta}(G)$ with $u\in V(Q_{j_1})$ and $v\in V(Q_{j_2})$ for some $j_1$ and $j_2$, $j_1\neq j_2$, or
\item [(c)] $(u,v)\in E(D_{1,\theta}(Q_{j_0}))$ for some $j_0$.
\end {itemize}
\end {thm}

\begin {proof} Suppose (a) holds. Then it follows from Lemma \ref {BP:L4a} that $(u,v)\in E(D_{1,\theta}(S_{\theta}(G)))$.

Suppose (b) holds. By  Theorem \ref {P:T5}, $u\in P_{\theta}(G\setminus A_{\theta}(G))$ and $v\in N_{\theta}(G\setminus A_{\theta}(G))$. By using part (a) of Theorem \ref {basic_property}, we  deduce that $\textnormal {mult} (\theta, Q_{j_1}\setminus u)=1$ and $\textnormal {mult} (\theta, Q_{j_2}\setminus v)=0$. Then by part (c) of Corollary \ref {S:C6}, $\textnormal {mult} (\theta, S_{\theta}(G)\setminus uv)=\textnormal {mult} (\theta, S_{\theta}(G))+\textnormal {mult} (\theta, Q_{j_1}\setminus u)+\textnormal {mult} (\theta, Q_{j_2}\setminus v)=\textnormal {mult} (\theta, S_{\theta}(G))+1$. Hence $(u,v)\in E(D_{1,\theta}(S_{\theta}(G)))$.

Suppose (c) holds. Then $\textnormal {mult} (\theta, Q_{j_0}\setminus uv)=1$. By part (b) of Corollary \ref {S:C6}, $
\textnormal {mult} (\theta, S_{\theta}(G)\setminus uv)=\textnormal {mult} (\theta, S_{\theta}(G))+\textnormal {mult} (\theta, Q_{j_0}\setminus uv)=\textnormal {mult} (\theta, S_{\theta}(G))+1$.  Hence  $(u,v)\in E(D_{1,\theta}(S_{\theta}(G)))$.

Suppose $(u,v)\in E(D_{1,\theta}(S_{\theta}(G)))$. By  Lemma \ref {BP:L16}, we may assume that  $u,v\in P_{\theta}(G)\cup N_{\theta}(G)$.  If $u,v\in V(Q_{j_0})$ for some $j_0$, then by part (b) of Corollary \ref {S:C6},  $\textnormal {mult} (\theta, S_{\theta}(G))+1=\textnormal {mult} (\theta, S_{\theta}(G)\setminus uv)=\textnormal {mult} (\theta, S_{\theta}(G))+\textnormal {mult} (\theta, Q_{j_0}\setminus uv)$, which implies that $\textnormal {mult} (\theta, Q_{j_0}\setminus uv)=1$, i.e., $(u,v)\in E(D_{1,\theta}(Q_{j_0}))$.

If $u\in V(Q_{j_1})$ and $v\in V(Q_{j_2})$ for some $j_1,j_2$, $j_1\neq j_2$, then by part (c) of Corollary \ref {S:C6},
$\textnormal {mult} (\theta, S_{\theta}(G))+1=\textnormal {mult} (\theta, S_{\theta}(G)\setminus uv)=\textnormal {mult} (\theta, S_{\theta}(G))+\textnormal {mult} (\theta, Q_{j_1}\setminus u)+\textnormal {mult} (\theta, Q_{j_2}\setminus v)$, which implies (without loss of generality) $\textnormal {mult} (\theta, Q_{j_1}\setminus u)=1$ and $\textnormal {mult} (\theta, Q_{j_2}\setminus v)=0$. By using part (a) of Theorem \ref {basic_property} again, we can deduce that $u\in P_{\theta}(G\setminus A_{\theta}(G))$ and $v\in N_{\theta}(G\setminus A_{\theta}(G))$. It then follows from Theorem \ref {P:T5}, that $u\in P_{\theta}(G)$ and $v\in N_{\theta}(G)$.
\end {proof}

\subsection{$D_{2,\theta}(G)$}

Using  Lemma \ref {BP:L4a}, one can easily deduce Lemma \ref {BP:L19}.
\begin {lm}\label {BP:L19} Let $G$ be a graph. If $(u,v)\in E(D_{2,\theta}(G))$ then either $u,v\in A_{\theta}(G)$ or $u\in A_{\theta}(G)$ and $v\in P_{\theta}(G)$ or $u,v\in P_{\theta}(G)$.
\end {lm}

\begin {thm}\label{d_graph_2_positive} Let $G$ be a graph and  $Q_1,\dots, Q_m$ be all the components in $G\setminus A_{\theta}(G)$ with \linebreak $\textnormal {mult} (\theta, Q_j)=0$ for all $j$. Then  $(u,v)\in E(D_{2,\theta}(S_{\theta}(G)))$ if and only if
\begin {itemize}
\item [(a)] $u,v\in  A_{\theta}(G)$, or
\item [(b)] $u\in A_{\theta}(G)$ and $v\in P_{\theta}(G)$, or
\item [(c)] $u,v\in P_{\theta}(G)$ with $u\in V(Q_{j_1})$ and $v\in V(Q_{j_2})$ for some $j_1$ and $j_2$, $j_1\neq j_2$, or
\item [(d)] $(u,v)\in E(D_{2,\theta}(Q_{i_0}))$ for some $i_0$.
\end {itemize}
\end {thm}

\begin {proof} Suppose (a) or (b) holds. Then it follows from Lemma \ref {BP:L4a} that $(u,v)\in E(D_{2,\theta}(S_{\theta}(G)))$.

Suppose (c) holds. By Theorem \ref {P:T5}, $u,v\in P_{\theta}(G\setminus A_{\theta}(G))$. By using part (a) of Theorem \ref {basic_property}, we  deduce that $\textnormal {mult} (\theta, Q_{j_1}\setminus u)=1=\textnormal {mult} (\theta, Q_{j_2}\setminus v)$. Then by part (c) of Corollary \ref {S:C6}, we have $\textnormal {mult} (\theta, S_{\theta}(G)\setminus uv)=\textnormal {mult} (\theta, S_{\theta}(G))+\textnormal {mult} (\theta, Q_{j_1}\setminus u)+\textnormal {mult} (\theta, Q_{j_2}\setminus v)=\textnormal {mult} (\theta, S_{\theta}(G))+2$. Hence $(u,v)\in E(D_{2,\theta}(S_{\theta}(G)))$.

Suppose (d) holds. Then $\textnormal {mult} (\theta, Q_{j_0}\setminus uv)=2$. By part (b) of Corollary \ref {S:C6}, $
\textnormal {mult} (\theta, S_{\theta}(G)\setminus uv)=\textnormal {mult} (\theta, S_{\theta}(G))+\textnormal {mult} (\theta, Q_{j_0}\setminus uv)=\textnormal {mult} (\theta, S_{\theta}(G))+2$.  Hence  $(u,v)\in E(D_{2,\theta}(S_{\theta}(G)))$.

Suppose $(u,v)\in E(D_{2,\theta}(S_{\theta}(G)))$. By  Lemma \ref {BP:L19}, we may assume that  $u,v\in P_{\theta}(G)$.  If $u,v\in V(Q_{j_0})$ for some $j_0$, then by part (b) of Corollary \ref {S:C6},  $\textnormal {mult} (\theta, S_{\theta}(G))+2=\textnormal {mult} (\theta, S_{\theta}(G)\setminus uv)=\textnormal {mult} (\theta, S_{\theta}(G))+\textnormal {mult} (\theta, Q_{j_0}\setminus uv)$, which implies that $\textnormal {mult} (\theta, Q_{j_0}\setminus uv)=2$, i.e., $(u,v)\in E(D_{2,\theta}(Q_{j_0}))$.

If $u\in V(Q_{j_1})$ and $v\in V(Q_{j_2})$ for some $j_1,j_2$, $j_1\neq j_2$, then by part (c) of Corollary \ref {S:C6},
$\textnormal {mult} (\theta, S_{\theta}(G))+2=\textnormal {mult} (\theta, S_{\theta}(G)\setminus uv)=\textnormal {mult} (\theta, S_{\theta}(G))+\textnormal {mult} (\theta, Q_{j_1}\setminus u)+\textnormal {mult} (\theta, Q_{j_2}\setminus v)$, which implies that $\textnormal {mult} (\theta, Q_{j_1}\setminus u)=1=\textnormal {mult} (\theta, Q_{j_2}\setminus v)$. As before using part (a) of Theorem \ref {basic_property}, we deduce that $u,v\in P_{\theta}(G\setminus A_{\theta}(G))$. It then follows from Theorem \ref {P:T5}, that $u,v\in P_{\theta}(G)$.
\end {proof}

Now let us look at the case $\theta=0$. Note that we have $N_{0}(G)=\varnothing$ for any graph $G$. By Lemma \ref {interlacing}, we have $\textnormal {mult} (0,G\setminus uv)=-2,0,2$ for all $u,v\in V(G)$. Hence,
\begin {thm}\label {BP:T22} Let $G$ be a graph. Then $D_{-1,0}(G)$ and $D_{1,0}(G)$ are  empty graphs.
\end {thm}


Now let us determine the edges of $D_{\theta}(G)$. We shall begin with the following lemma.

\begin {lm}\label {R:L3} Let $G$ be a graph and  $u\in P_{\theta}(G)\cup N_{\theta}(G)$. Then $A_{\theta}(G)\subseteq A_{\theta}(G\setminus u)$.
\end {lm}

\begin {proof} Let $w\in A_{\theta}(G)$. Then by Theorem \ref {P:T5}, $\textnormal {mult} (\theta, G\setminus wu)=\textnormal {mult} (\theta,G)+2$ or $\textnormal {mult} (\theta,G)+1$, depending on whether $u\in P_{\theta}(G)$ or $u\in N_{\theta}(G)$. In either cases, $w\notin B_{\theta}(G\setminus u)$. Let $z\in B_{\theta}(G)$ be adjacent to $w$. By Lemma \ref{godsil_positive} and Lemma \ref{Ku-Chen-neutral}, we have, $z\in B_{\theta}(G\setminus u)$. This implies that $w\in A_{\theta}(G\setminus u)$ and $A_{\theta}(G)\subseteq A_{\theta}(G\setminus u)$.
\end {proof}

\begin {thm}\label{extreme_in_G} Let $G$ be a graph and  $u,v\in P_{\theta}(G)\cup N_{\theta}(G)$. Then $A_{\theta}(G)$ is an $\theta$-extreme set in $G\setminus uv$.
\end {thm}

\begin {proof} By Lemma \ref {R:L3},  $A_{\theta}(G)\subseteq A_{\theta}(G\setminus u)$. If $v\in P_{\theta}(G\setminus u)\cup N_{\theta}(G\setminus u)$ then by Lemma \ref {R:L3}, $A_{\theta}(G\setminus u)\subseteq A_{\theta}(G\setminus uv)$.  If $v\in A_{\theta}(G\setminus u)$, by Theorem \ref {P:T5}, $A_{\theta}(G)\subseteq A_{\theta}(G\setminus uv)$. In either cases we have $A_{\theta}(G)$ is an $\theta$-extreme set in $G\setminus uv$.

So we may assume $v\in B_{\theta}(G\setminus u)$. Using Lemma \ref{Ku-Chen-neutral}, we deduce that $u\in P_{\theta}(G)$. So $\textnormal {mult} (\theta, G\setminus u)=\textnormal {mult} (\theta,G)+1$ and by  Theorem
\ref {P:T5}, $\textnormal {mult} (\theta, (G\setminus u)\setminus A_{\theta}(G))=\textnormal {mult} (\theta,G)+1+\vert A_{\theta}(G)\vert$. Again by Theorem \ref {P:T5}, we see that $v\in B_{\theta}((G\setminus u)\setminus A_{\theta}(G))$. Therefore
\begin {equation}
\textnormal {mult} (\theta, (G\setminus uv)\setminus A_{\theta}(G))=\textnormal {mult} (\theta,G)+\vert A_{\theta}(G)\vert.\notag
\end {equation}
Since  $\textnormal {mult} (\theta, G\setminus uv)=\textnormal {mult} (\theta,G)$, $A_{\theta}(G)$ is an $\theta$-extreme set in $G\setminus uv$.
\end {proof}

\begin {cor}\label {S:C3b}  Let $G$ be a graph. Let $Q_1,\dots, Q_m$ be all the components in $G\setminus A_{\theta}(G)$ with \linebreak $\textnormal {mult} (\theta, Q_j)=0$ for all $j$.
Then the following holds:
\begin {itemize}
\item [(a)] If $u,v\in V(Q_{j_0})$ for some $j_0$, then
\begin {equation}
\textnormal {mult} (\theta, G\setminus uv)=\textnormal {mult} (\theta, G)+\textnormal {mult} (\theta, Q_{j_0}\setminus uv).\notag
\end {equation}
\item [(b)] If $u\in V(Q_{j_1})$ and $v\in V(Q_{j_2})$ for some $j_1,j_2$, $j_1\neq j_2$, then
\begin {equation}
\textnormal {mult} (\theta, G\setminus uv)=\textnormal {mult} (\theta, G)+\textnormal {mult} (\theta, Q_{j_1}\setminus u)+\textnormal {mult} (\theta, Q_{j_2}\setminus v).\notag
\end {equation}
\end {itemize}
\end {cor}

\begin {proof} (a) Suppose $u,v\in V(Q_{j_0})$ for some $j_0$. By  Theorem \ref {extreme_in_G}, $A_{\theta}(G)$ is an $\theta$-extreme set in $G\setminus uv$. Therefore
\begin {equation}
\textnormal {mult} (\theta, G\setminus (A_{\theta}(G)\cup \{u,v\}))=\textnormal {mult} (\theta, G\setminus uv)+\vert A_{\theta}(G)\vert.\notag
\end {equation}

Let $H_1,\dots, H_q$ be all the $\theta$-critical components in $G\setminus A_{\theta}(G)$.  By  Corollary \ref {P:C7}, and part (a) of Theorem \ref {basic_property}, we have
 \begin {align}
\textnormal {mult} (\theta, G\setminus (A_{\theta}(G)\cup \{u,v\})) &=\textnormal {mult} (\theta, Q_{j_0}\setminus uv)+\sum_{1\leq i\leq q} \textnormal {mult} (\theta, H_i)\notag\\
&=\textnormal {mult} (\theta, Q_{j_0}\setminus uv)+\textnormal {mult} (\theta, G)+\vert A_{\theta}(G)\vert.\notag
\end {align}
 This implies that $\textnormal {mult} (\theta, G\setminus uv)=\textnormal {mult} (\theta, G)+\textnormal {mult} (\theta, Q_{j_0}\setminus uv)$.

(b) is proved similarly.
\end {proof}

\begin {thm}\label{d_graph_for_G} Let $G$ be a graph and  $Q_1,\dots, Q_m$ be all the components in $G\setminus A_{\theta}(G)$ with \linebreak $\textnormal {mult} (\theta, Q_j)=0$ for all $j$. Then  $(u,v)\in E(D_{\theta}(G))$ if and only if
\begin {itemize}
\item [(a)] $u\in  B_{\theta}(G)$ and $v\in V(G)$, or
\item [(b)] $u,v\in N_{\theta}(G)$ with $u\in V(Q_{j_1})$ and $v\in V(Q_{j_2})$ for some $j_1$ and $j_2$, $j_1\neq j_2$, or
\item [(c)] $(u,v)\in E(D_{0,\theta}(Q_{j_0}))$ for some $j_0$.
\end {itemize}
\end {thm}

\begin {proof} Suppose (a) holds. Since $u\in B_{\theta}(G)$, $\textnormal {mult} (\theta, G\setminus u)=\textnormal {mult} (\theta, G)-1$. By Lemma \ref {interlacing}, we have $\textnormal {mult} (\theta, G\setminus uv)\leq \textnormal {mult} (\theta, G)$ for all $v\in V(G)$. Hence $(u,v)\in E(D_{\theta}(G))$ for all $v\in V(G)$.

Suppose (b) holds. By  Theorem \ref {P:T5}, $u,v\in N_{\theta}(G\setminus A_{\theta}(G))$. By using part (a) of Theorem \ref {basic_property}, we can deduce that $\textnormal {mult} (\theta, Q_{j_1}\setminus u)=0=\textnormal {mult} (\theta, Q_{j_2}\setminus v)$. By part (b) of Corollary \ref {S:C3b}, $\textnormal {mult} (\theta, G\setminus uv)=\textnormal {mult} (\theta, G)+\textnormal {mult} (\theta, Q_{j_1}\setminus u)+\textnormal {mult} (\theta, Q_{j_2}\setminus v)=\textnormal {mult} (\theta, G)$. Hence $(u,v)\in E(D_{\theta}(G))$.

Suppose (c) holds. Then $\textnormal {mult} (\theta, Q_{j_0}\setminus uv)=0$. By part (a) of Corollary \ref {S:C3b}, $\textnormal {mult} (\theta, G\setminus uv)=\textnormal {mult} (\theta, G)+\textnormal {mult} (\theta, Q_{j_0}\setminus uv)=\textnormal {mult} (\theta, G)$. Hence $(u,v)\in E(D_{\theta}(G))$.

Suppose $(u,v)\in E(D_{\theta}(G))$. By Lemma \ref {BP:L6a}, Lemma \ref {BP:L10} and Lemma \ref {BP:L13}, we may assume that $u,v\in P_{\theta}(G)\cup N_{\theta}(G)$. Suppose $u,v\in V(Q_{j_0})$ for some $j_0$. By part (a) of Corollary \ref {S:C3b}, $\textnormal {mult} (\theta, G\setminus uv)=\textnormal {mult} (\theta, G)+\textnormal {mult} (\theta, Q_{j_0}\setminus uv)$. Since $\textnormal {mult} (\theta, G\setminus uv)\leq \textnormal {mult} (\theta, G)$, we must have
$\textnormal {mult} (\theta, Q_{j_0}\setminus uv)=0$ and  $(u,v)\in E(D_{0,\theta}(Q_{j_0}))$.

Suppose  $u\in V(Q_{j_1})$ and $v\in V(Q_{j_2})$ for some $j_1$ and $j_2$, $j_1\neq j_2$. By part (b) of Corollary \ref {S:C3b}, $\textnormal {mult} (\theta, G\setminus uv)=\textnormal {mult} (\theta, G)+\textnormal {mult} (\theta, Q_{j_1}\setminus u)+\textnormal {mult} (\theta, Q_{j_2}\setminus v)$. Since $\textnormal {mult} (\theta, G\setminus uv)\leq \textnormal {mult} (\theta, G)$, we must have
$\textnormal {mult} (\theta, Q_{j_1}\setminus u)=0=\textnormal {mult} (\theta, Q_{j_2}\setminus v)$. This implies that $u,v\in N_{\theta}(G\setminus A_{\theta}(G))$. Hence by Theorem \ref {P:T5}, $u,v\in N_{\theta}(G)$.
\end {proof}

Note that in Theorem \ref{d_graph_for_G}, the edge-set in $D_{\theta}(G)$ depends only on the Gallai-Edmonds decomposition of $G$. Therefore if $G$ and $G'$ have the same Gallai-Edmonds decomposition with respect to $\theta$ via $\psi$, then $D_{\theta}(G)\overset{\psi}{\cong} D_{\theta}(G')$. Since $G$ and $S_{\theta}(G)$ have the same Gallai-Edmonds decomposition via the identity map, we have $D_{\theta}(G)=D_{\theta}(S_{\theta}(G))$. This proves the following corollary.

\begin {cor}\label{gallai_Edmond_decomposition_G_S} If $G$ and $G'$ have the same Gallai-Edmonds decomposition with respect to $\theta$, then $D_{\theta}(G)\cong D_{\theta}(G')$. In particular, $D_{\theta}(G)=D_{\theta}(S_{\theta}(G))$.
\end {cor}

Note that if $G$ is a graph with $n$ vertices then $E(K_n)=E(D_{-2,\theta}(G))\cup E(D_{-1,\theta}(G))\cup E(D_{0,\theta}(G))\cup E(D_{1,\theta}(G))\cup E(D_{2,\theta}(G))$, where $K_n$ is the complete graph on $n$ vertices ($V(K_n)=V(G)$).

If we denote the complement of a graph $G$ by $\overline G$, by Corollary \ref {gallai_Edmond_decomposition_G_S}, we have

\begin {cor}\label {R:C6} Let $G$ be a graph. Then $\overline {D_{\theta}(G)}=\overline {D_{\theta}(S_{\theta}(G))}=G_{+}$, where $G_+$ is the graph with $V(G_+)=V(G)$ and $E(G_+)=E(D_{1,\theta}(G))\cup E(D_{2,\theta}(G))$.
\end {cor}

\section{$\theta$-Nice Sets and Matchings}

In this section, we first relate $\theta$-nice sets with matchings. Then we proceed to show that $D_{\theta}(G)$ always contain certain induced subgraphs of $G$ related to $\theta$.

Recall that a path $P$ is called $\theta$-essential if $\m (\theta, G\setminus P)=\m (\theta, G)-1$. We shall require the following lemmas:

\begin{lm}\label{essential-path}\textnormal {\cite[Lemma 3.3]{G}}
If $P$ is a $\theta$-essential path in $G$, then both of its end points are $\theta$-essential in $G$.
\end{lm}

\begin{lm}\label{essential-neighbor}\textnormal {\cite[Lemma 3.4]{G}}
Let $G$ be a graph and $u$ a vertex in $G$ which is not $\theta$-essential. Then $u$ is $\theta$-positive in $G$ if and only if some neighbor of it is $\theta$-essential in $G \setminus u$.
\end{lm}

\begin{lm}\label{nonpositive_path}
Let $u, v$ be two distinct $\theta$-positive vertices of $G$. Then $\m(\theta, G \setminus uv) \le \m(\theta, G)$ if and only if there exists a path $P$ from $u$ to $v$ such that $\m(\theta, G \setminus P) \le \m(\theta, G)$.
\end{lm}

\begin{proof}
Let $k=\m(\theta, G)$ where $k \ge 0$. Consider the Heilmann-Lieb Identity (see \cite[Theorem 6.3]{HL} and \cite[Lemma 2.4]{G}):
\[ \mu(G \setminus u,x)\mu(G \setminus v,x) - \mu(G,x)\mu(G \setminus uv,x) = \sum_{P \in \mathbb{P}(u,v)} \mu(G \setminus P,x)^{2}\]
where $\mathbb{P}(u,v)$ denote the set of paths from $u$ to $v$ in $G$.

$(\Longrightarrow)$ Suppose there is no path $P$ from $u$ to $v$ such that $\m(\theta, G \setminus P) \le \m(\theta, G)$. Then $\theta$ is a root of the polynomial $\mu(G \setminus u,x)\mu(G \setminus v,x) - \sum_{P \in \mathbb{P}(u,v)} \mu(G \setminus P,x)^{2}$ with multiplicity at least $2k+2$. But this contradicts the fact that the multiplicity of $\theta$ as a root of $\mu(G,x)\mu(G \setminus uv,x)$ is at most $2k$.

$(\Longleftarrow)$ Suppose $\m(\theta, G \setminus uv) > \m(\theta, G)=k$. By Lemma \ref{godsil_positive}, $\m(\theta, G \setminus uv)=k+2$. Since $\{P \in \mathbb{P}(u,v): \m(\theta, G \setminus P) \le k\} \not = \emptyset$, we can write
\[ \sum_{P \in \mathbb{P}(u,v) \atop \m(\theta, G \setminus P) \le k} \mu(G \setminus P,x)^{2} = \sum_{i=1}^{m} (x-\theta)^{2t} (g_{i}(x))^{2}\]
for some $m$ and $t \le k$ and $g_{j}(\theta) \not = 0$ for some $j \in \{1, \ldots, m\}$.

On the other hand, from the Heilmann-Lieb Identity, we see that
\[ \mu(G \setminus u,x)\mu(G \setminus v,x) - \mu(G,x)\mu(G \setminus uv,x)- \sum_{P \in \mathbb{P}(u,v) \atop \m(\theta, G \setminus P)> k} \mu(G \setminus P,x)^{2} =  \sum_{P \in \mathbb{P}(u,v) \atop \m(\theta, G \setminus P) \le k} \mu(G \setminus P,x)^{2} \]
where the left-hand side has $\theta$ as a root with multiplicity at least $2k+2$. Therefore
\[ \frac{1}{(x-\theta)^{2t}} \left(\mu(G \setminus u,x)\mu(G \setminus v,x) - \mu(G,x)\mu(G \setminus uv,x)- \sum_{P \in \mathbb{P}(u,v) \atop \m(\theta, G \setminus P)> k} \mu(G \setminus P,x)^{2} \right) = \sum_{i=1}^{m} (g_{i}(x))^{2}\]
where the left-hand side has $\theta$ as a root with nonzero multiplicity. But this contradicts the fact that $\sum_{i=1}^{m} (g_{i}(\theta))^{2}>0$.
\end{proof}

\begin{thm}\label{nice_matching}
Suppose $X=\{x_{1}, \ldots, x_{m}\}$ is $\theta$-nice in $G$ and $\m(\theta, G)=k$ (We allow $k$ to take zero value). Then there exists a set $Y=\{y_{1}, \ldots, y_{m}\}$ disjoint from $X$ such that
\begin{itemize}
\item[(i)] $M=\{x_{1}y_{1}, \ldots, x_{m}y_{m}\}$ is a matching of size $m$ in $G$,
\item[(ii)] for any $M' \subseteq M$, we have $\m(\theta, G \setminus V(M')) = k$ and if $|X \setminus V(M')| \ge 2$, then $X \setminus V(M')$ is $\theta$-nice in $G \setminus V(M')$, and
\item[(iii)] $Y$ is an independent set.
\end{itemize}
\end{thm}

\begin{proof} We shall prove it by induction on $m$. Suppose $m=2$. By Lemma \ref{essential-neighbor}, $x_1$ is adjacent to a vertex $y_1$ which is $\theta$-essential in $G \setminus x_1$. Therefore $\m (\theta,G\setminus x_1y_1)=k$. Note that $\m (\theta,G\setminus x_1x_2)=k+2$. So by Lemma \ref{interlacing}, $\m (\theta,G\setminus x_1x_2y_1)\geq k+1$, and $x_2$ is $\theta$-positive in $G\setminus x_1y_1$. Again by Lemma \ref{essential-neighbor}, $x_2$ is adjacent to a vertex $y_2$ in $G\setminus x_1y_1$ and $y_2$ is $\theta$-essential in $G \setminus x_1y_1x_2$. Hence $\m (\theta,G\setminus x_1y_1x_2y_2)=k$. Now part (i) has been proved. For part (ii), if $M'=M$ or $M'=\{x_1y_1\}$, we are done. Suppose $M'=\{x_2y_2\}$. Since $x_2$ is $\theta$-positive in $G$, by Lemma \ref{interlacing}, $\m (\theta, G\setminus x_2y_2)\geq k$. If the equality holds, we are done. Suppose $\m (\theta, G\setminus x_2y_2)\geq k+1$. Then by Lemma \ref{interlacing}, we deduce that $\m (\theta, G\setminus x_2y_2)$ is either equal to $k+2$ or $k+1$. If the former holds then by Lemma \ref{path_interlacing},  $\m (\theta, G\setminus x_2y_2x_1y_1)\geq k+1$, a contrary to the fact that  $\m (\theta, G\setminus x_2y_2x_1y_1)= k$. Suppose the latter holds. Note that  $\m (\theta, G\setminus x_2x_1)=k+2$. By Lemma \ref{interlacing}, $\m (\theta, G\setminus x_2x_1y_2)\geq k+1$. So $x_1$ is either $\theta$-neutral or $\theta$-positive in $G\setminus x_2y_2$. By Lemma \ref{path_interlacing} and Lemma \ref{essential-path}, $\m (\theta, G\setminus x_2y_2x_1y_1)\geq k+1$, a contradiction. Hence $\m (\theta, G\setminus x_2y_2)=k$ and the proof for part (ii) for $m=2$ is complete.

Let $m\geq 3$. Assume that it is true for all $\theta$-nice set $X'$ with $\vert X'\vert<\vert X\vert$. As before,
 $x_1$ is adjacent to a vertex $y_1$ which is $\theta$-essential in $G \setminus x_1$. Therefore $\m (\theta,G\setminus x_1y_1)=k$. On the other hand, by Theorem \ref{general_nice_case}, $X$ is an $\theta$-extreme set. So $\m (\theta,G\setminus X)=k+\vert X\vert$. Let $X'=\{x_2,x_3,\dots, x_m\}$. By Lemma \ref{interlacing},
\begin{equation}
k+\vert X\vert-1=k+\vert X'\vert\geq \m (\theta,(G\setminus x_1y_1)\setminus X')=\m (\theta,G\setminus (X\cup \{y_1\}))\geq k+\vert X\vert-1.\notag
\end{equation}
 Thus  $\m (\theta,(G\setminus x_1y_1)\setminus X')=k+\vert X'\vert$ and $X'$ is an $\theta$-extreme set in $G\setminus x_1y_1$. Note that $X'$ is a $\theta$-nice set by Theorem \ref{general_nice_case}. Therefore by induction, there is a matching $M_1=\{x_2y_2,x_3y_3,\dots, x_my_m\}$ in $G\setminus x_1y_1$ for which the conclusions in part (ii)  holds. Let $M=M_1\cup\{x_1y_1\}$. Then part (i) is proved.

Let $M'\subseteq M$. Suppose $x_1y_1\in M'$. Let $M_1'=M'\setminus \{x_1y_1\}$. Then we have $\m (\theta, G\setminus V(M'))=\m (\theta, (G\setminus x_1y_1)\setminus V(M_1'))=k$, where the last inequality follows from induction. Furthermore $X\setminus V(M')=X'\setminus V(M_1')$, so, if $\vert X\setminus V(M')\vert\geq 2$, $X \setminus V(M')$ is $\theta$-nice in $G \setminus V(M')$.

Suppose $x_1y_1\notin M'$. Let $X_2=X\setminus V(M')$. Since $X$ is an $\theta$-extreme set, by Lemma \ref{interlacing}, it is not hard to deduce that $\m(\theta, G \setminus (X\setminus X_2))=k+\vert X\setminus X_2\vert$.  By Lemma \ref{interlacing} again, $\m(\theta, G \setminus V(M'))=\m(\theta, (G \setminus (X\setminus X_2))\setminus (V(M')\setminus (X\setminus X_2)))\geq k+\vert X\setminus X_2\vert-\vert X\setminus X_2\vert=k$.

Suppose $\m(\theta, G \setminus V(M'))\geq k+1$. If $\m(\theta, G \setminus V(M'))\geq k+2$, then by Lemma \ref{path_interlacing}, we have $\m(\theta, (G\setminus V(M'))\setminus x_1y_1)\geq k+1$, a contradiction, for by induction we have $\m(\theta, (G\setminus x_1y_1) \setminus V(M'))=k$. Thus $\m(\theta, G \setminus V(M'))=k+1$. Let $X_3=X\setminus X_2$. Since $X$ is an $\theta$-extreme set, by Lemma \ref{interlacing},  $\m(\theta, G \setminus (X_3\cup\{x_1\}))=k+\vert X_3\vert +1$. Again by Lemma \ref{interlacing}, $\m(\theta, (G \setminus (X_3\cup\{x_1\}))\setminus V(M'))\geq k+1$. Note that $(G \setminus (X_3\cup\{x_1\}))\setminus V(M')=G\setminus (V(M')\cup \{x_1\})$. So $x_1$ is either $\theta$-neutral or $\theta$-positive in $G\setminus V(M')$. But then by  Lemma \ref{path_interlacing} and Lemma \ref{essential-path}, $\m (\theta, G\setminus (V(M')\cup \{x_1,y_1\}))\geq k+1$, a contradiction. Hence $\m (\theta, G\setminus V(M'))=k$.

Suppose $|X_2| \ge 2$. Recall that $X$ is an $\theta$-extreme set. So $\m (\theta, G\setminus X)=k+\vert X\vert$ and by Lemma \ref{interlacing}, $\m (\theta, (G\setminus X)\setminus V(M'))\geq k+\vert X\vert-\vert X\setminus X_2\vert=k+\vert X_2\vert$. Note that $(G\setminus X)\setminus V(M')=(G\setminus V(M'))\setminus X_2$. By Lemma \ref{interlacing} again, $\m (\theta, (G\setminus V(M'))\setminus X_2)\leq k+\vert X_2\vert$. Hence $X_2$ is an $\theta$-extreme set and thus a $\theta$-nice set in $G\setminus V(M')$. This completes the proof of part (ii).

Suppose $Y$ is not an independent set, i.e. $y_{i}$ is joined to $y_{j}$ for some $i, j \in \{1, \ldots, m\}$. Then the path $P:=x_{i}y_{i}y_{j}x_{j}$ satisfies $\m(\theta, G \setminus P)=\m(\theta, G)$ by part (ii). By Lemma \ref{nonpositive_path}, we deduce that $\m(\theta, G \setminus x_{i}x_{j}) \le \m(\theta, G)$, contradicting the $\theta$-niceness of $X$.
\end{proof}

\begin{thm}
Let $X$ be a $\theta$-nice set in $G$ and $Y$ be a corresponding independent set guaranteed by Theorem \ref{nice_matching}. Then $D_{\theta}(G)$ contains an isomorphic copy of the subgraph of $G$ induced by $X \cup Y$.
\end{thm}

\begin{proof}
Let $X=\{x_{1}, \ldots, x_{m}\}$ and $Y=\{y_{1}, \ldots, y_{m}\}$. Consider the subgraph $H$ of $G$ induced by $X \cup Y$.

By part (ii) of Theorem \ref{nice_matching}, $(x_{i},y_{i}) \in E(D_{\theta}(G))$ for all $i=1, \ldots, m$.

If $(x_{i}, x_{j}) \in E(H)$, then the path $P:=y_{i}x_{i}x_{j}y_{j}$ satisfies $\m(\theta, G \setminus P) \le \m(\theta, G)$ by part (ii) of Theorem \ref{nice_matching}, so by Lemma \ref{nonpositive_path}, $(y_{i}, y_{j}) \in E(D_{\theta}(G))$.

Similarly, if $(x_{i}, y_{j}) \in E(H)$ then the path $Q:=y_{i}x_{i}y_{j}x_{j}$ satisfies $\m(\theta, G \setminus Q) \le \m(\theta, G)$, whence $(y_{i}, x_{j}) \in E(D_{\theta}(G))$.

Therefore, $D_{\theta}(G)$ contains an isomorphic copy of $H$.
\end{proof}

\end{document}